\documentclass[a4paper, 10pt, reqno, final]{article}

\usepackage[T1]{fontenc}		     % codifica dei font
\usepackage[utf8]{inputenc}			 % codifica di input
\usepackage[american]{babel}
\usepackage{float}
\usepackage{caption}
\usepackage{subcaption}
\usepackage{cancel}
\usepackage{amsmath}
\usepackage{amssymb}
\usepackage{amsthm}
\usepackage{tikz-cd}
\usetikzlibrary{calc} 
\usepackage{multicol}
\usepackage{bbold}
	\theoremstyle{plain}
		\newtheorem{mainthm}{\textsc{Theorem}}		
				
				\newtheorem{thm}{Theorem}[section]	
						\newtheorem{cor}[thm]{Corollary}	
				
				\newtheorem{lem}[thm]{Lemma}		
						\newtheorem{prop}[thm]{Proposition}

	\theoremstyle{definition}
		\newtheorem{defn}[thm]{Definition}	
				\newtheorem{ex}[thm]{Example}		
					\theoremstyle{remark}
		\newtheorem{rem}[thm]{Remark}		
				\newtheorem{note}[thm]{Notation}		
				\numberwithin{equation}{section}
				\usepackage{mathtools}		
				\usepackage{mathrsfs}		
				
\usepackage{eucal}			
\usepackage{soul}
\usepackage{braket}			
\usepackage[all,pdf]{xy}		

\usepackage{mparhack}		% fix bug for \marginpar
\usepackage{relsize}			% dimensioni testo

\usepackage{a4wide}		% aumenta i margini di pagina

\usepackage{booktabs}		% per fare bene le tabelle
\usepackage{multirow}		% per celle su pi righe nelle tabelle
\usepackage{caption}		% per le didascalie di tabelle e figure
\usepackage{rotating}

\usepackage{varioref}		% per indicazione di pagina nei riferimenti
\usepackage{footmisc}
\usepackage{multicol}
\usepackage{graphicx}		% per le figure
\usepackage{epsfig}
\usepackage{enumerate}		% per impostare facilmente gli elenchi numerati
\usepackage[colorlinks=true, linkcolor=black, citecolor=black,
urlcolor=blue]{hyperref}		
\mathtoolsset{showonlyrefs}
\newcommand{\nullity}{\mathrm{n}_0\,}
\newcommand{\noo}[1]{\overset {\mbox{%
\lower1pt\hbox{${\scriptstyle o}$}}}{\mathrm n}_{\mbox{%
\lower2pt\hbox{$\scriptstyle#1$}}}}
\newcommand{\coindex}{\mathrm{n}_+\,}
\newcommand{\itriple}{\iota}
\newcommand{\trasp}[1]{{#1}^\top}
\newcommand{\iMor}{\mathrm{n_{-}}}		% indice di Morse
\newcommand{\R}{\mathbb{R}}		
\newcommand{\ZZ}{\mathbf{Z}}
\newcommand{\Sym}{\mathrm{S}}
\newcommand{\Mat}{\mathrm{M}}
\newcommand{\Gr}{\mathcal{G}}
\newcommand{\iCLM}{\mu^{\scriptscriptstyle{\mathrm{CLM}}}}

\DeclareMathOperator{\sgn}{sgn}		% signature

\newcommand{\coiMor}{\mathrm{n_+}}
\newcommand{\Id}{I}
\DeclareMathOperator{\spfl}{sf}			% spectral flow
	
\DeclareMathOperator{\CFs}{\mathscr{CF}^{s}}

\newcommand{\Lagr}{\Lambda}

\newcommand{\cfsa}{\mathcal{CF}^{sa}}

\newcommand{\GStar}[1]{{\mathcal S}_{#1}}

\newcommand{\irel}{I}
\newcommand{\iindex}[1]{\mu_{\scriptscriptstyle{\mathrm{Mor}}}\left[#1\right]}
\newcommand{\coiindex}[1]{\mathrm{n_+}\left[#1\right]}
\newcommand{\Real}{\mathrm{Re}}
   
\newcommand{\Imm}{\mathrm{Im}}

\newcommand{\norm}[1]{\lVert#1\rVert}

\def\sq(H3){\sqrt{-1}}

\DeclareMathOperator{\rk}{rank}	
\def\={:=}
\def\>{\supset}
\def\<{\subset}

\def\12{\dfrac{1}{2}}
\def\0{^{\circ}}

\def\RR{{\mathbb R}}

\def\ZZ{{\mathbb Z}}

\def\m{\mu}

\def\R{\RR}

\def\Z{\ZZ}

\def\Cl{\mbox{\rm C$\ell$}}

 \DeclareMathOperator{\dist}{dist}

\DeclareMathOperator{\dom}{D}

\DeclareMathOperator{\Graph}{gr}
\newcommand{\Grn}[1]{\mathcal G({#1})}
\DeclareMathOperator{\Grass}{\mathcal S} 

\DeclareMathOperator{\image}{im}

\DeclareMathOperator{\Lag}{\mathrm{Lag}} 
\DeclareMathOperator{\Lin}{\mathscr{B}}

\DeclareMathOperator{\Sp}{Sp} 
\renewcommand{\Cl}{\mathscr{Cl}}

\DeclareMathOperator{\ev}{ev} 
\DeclareMathOperator{\Tr}{Tr}

\usepackage{bm}

\makeatletter
\def\namedlabel#1#2{\begingroup
    #2%
    \def\@currentlabel{#2}%
    \phantomsection\label{#1}\endgroup
}
\makeatother

\title{Index theory for non-compact quantum graphs}
\author{Daniele Garrisi, Alessandro Portaluri
\thanks{A.P. is partially supported by Tamkeen Portaluri 2025--26 Faculty Research Funds, Tamkeen NYU Abu Dhabi, and by GNAMPA--INDAM, Italy.}, Li Wu\thanks{L.W. is partially supported by  supported by the  National Natural Science Foundation of China NSFC N.12171281.}}

\date{\today}

\begin{document}
 \maketitle

\begin{abstract}
We develop an index theory for variational problems on non-compact quantum graphs. The main results are a spectral flow formula, relating the net change of eigenvalues to the Maslov index of boundary data, and a Morse index theorem, equating the negative directions of the Lagrangian action with the total multiplicity of conjugate instants along the edges.  

These results extend classical tools in global analysis and symplectic geometry to graph-based models. The spectral flow formula is proved by constructing a Lagrangian intersection theory in the Gelfand–Robbin quotients of the second variation of the action. 

This approach also recovers, in a unified way, the known formulas for heteroclinic, half-clinic, homoclinic, and bounded orbits of (non)autonomous Lagrangian systems.

\end{abstract}

\section{Introduction}
Quantum graphs are mathematical structures that consist of a graph equipped with differential operators acting along its edges and suitable conditions at the vertices. Originally introduced as simplified models of quantum mechanical systems with constrained geometry, quantum graphs have become a versatile framework for studying wave propagation in branched structures. Applications arise in mesoscopic physics (e.g., quantum wires and nanotechnology), optics (e.g., photonic crystals), chemistry (e.g., the theory of conjugated molecules), and increasingly in biology and data science \cite{BK13, Kuc04, Exn08}.

A typical quantum graph consists of a set of vertices $V$ and edges $E$, where each edge $e \in E$ is associated with an interval $I_e \subset \mathbb{R}$ (possibly infinite), and a differential operator $L_e$ acting on functions defined on $I_e$. These operators are typically self-adjoint, and the vertex conditions ensure that the global operator is self-adjoint on the whole graph. The spectral theory of such operators provides insight into resonance phenomena, scattering, wave localization, and the dynamics of nonlinear evolution equations.
 
% \subsection*{Motivation and Relevance}
The analysis of differential equations on graphs is motivated by two parallel developments. First, from a physical viewpoint, branched systems are ubiquitous in modern science and technology. The behavior of waves in such systems can often be reduced to one-dimensional models with non-trivial coupling conditions at junctions. Quantum graphs provide a natural abstraction of such problems.

Second, from a mathematical standpoint, quantum graphs offer a rich interplay between spectral theory, functional analysis, variational calculus, and symplectic geometry. Many foundational ideas in differential geometry—such as the Morse index theorem, spectral flow, and Maslov index—have analogs in the setting of quantum graphs, but require significant adaptations due to singularities, infinite domains, and the non-manifold nature of graphs. 

While the theory of compact quantum graphs is now mature, less is known in the non-compact case. Graphs with infinite edges (e.g., rays or half-lines) arise naturally in models of open systems and waveguides. They present additional analytical difficulties related to essential spectra, lack of compact embeddings, and boundary behavior at infinity \cite{HP17, LSS18, LS20, HPW25} and references therein.

Our goal in this paper is to build a comprehensive index theory for non-compact quantum graphs. Specifically, we study variational problems defined by action functionals on graphs, characterize their critical points via Euler-Lagrange equations, and analyze the second variation to understand stability. To do this, we develop an abstract framework using symplectic reduction and Gelfand-Robbin quotients to generalize the classical Morse index and spectral flow formulas.

\subsection*{Structure of the Paper and Main Contributions}
This work makes four main contributions:
\begin{enumerate}
    \item We formulate a variational framework for differential equations on non-compact quantum graphs, including appropriate Sobolev spaces, Lagrangian functionals, and Euler-Lagrange boundary value problems.
    \item We prove a \emph{spectral flow formula} for a family of self-adjoint Sturm--Liouville operators on a graph. The spectral flow is expressed in terms of a Maslov index associated to a path of Lagrangian subspaces in a finite-dimensional symplectic quotient space.
    \item We establish a \emph{Morse index theorem}, relating the Morse index of a critical point of the action functional to the number of conjugate instants—i.e., times where nontrivial solutions to linearized boundary problems exist.
    \item We apply these results to nonlinear Schrödinger equations (NLS) on graphs and obtain stability criteria for standing wave solutions.
\end{enumerate}

The core of the paper is given by Section~\ref{sec:Morse-index-thm} and Section~\ref{sec:star-graphs} where we develop the analytic and symplectic machinery necessary to extend Morse index theory and spectral flow to the context of quantum graphs, which may include non-compact edges (e.g., half-lines). The ultimate goal is to compute the Morse index of variational problems on such graphs and relate it to a Maslov index associated to the boundary conditions.

Let $\Gr$ be a non-compact graph, and let $s \mapsto L_s$ denote a gap-continuous family of self-adjoint extensions of a Sturm--Liouville operator defined on $\Gr$. We recall that
the space of closed and unbounded operators can be equipped with the metric \(d(L_{s_1},L_{s_2}):=\hat{\delta}(\Graph(L_{s_1}),\Graph(L_{s_2}))\), where \(\hat{\delta}\) is the gap between two subspaces defined in \cite[Ch.~IV,~\S2.1]{Kat80}.

The essential objects introduced include:
\begin{itemize}
  \item The maximal operator $L^*_s$ and minimal domain $D(L_s)$;
  \item A decomposition formula $ D(L^*_s)=D(L_s)\oplus U$, where $U$ is an even dimensional closed subspace of $H$  that inherits a  symplectic structure from the Green's identity
  \item A Lagrangian subspace $\Lambda_s \subset \beta$ prescribing the boundary conditions;
  \item The Cauchy data space $V_s = \ker L^*_s$;
  \item A trace map $\Tr$ identifying $U$ with $\mathbb{R}^{2k}$ for some $k$.
\end{itemize}

\noindent
\paragraph{A Spectral Flow Formula on (non-compact) graphs.} Let $\{s\mapsto L_{s, \Lambda_s}\}$ be a gap-continuous path  of self-adjoint extensions with $\Lambda_s \in \mathrm{Lag}(U)$.  Then,
\[
\mathrm{sf}(L_{s,\Lambda_s};\ s\in[0,1]) = -\iCLM(\Tr(\Lambda_s), \Tr(p(V_s));\ s\in[0,1]),
\]
where $p(V_s)$ is the projection of the Cacuhy data space onto $U$, and $\iCLM$ denotes the Cappell--Lee--Miller Maslov index.

This result reduces the problem of tracking eigenvalue crossings of $L_{s,\Lambda_s}$  to a topological count of Lagrangian intersections in a finite-dimensional symplectic space. (Cfr. Theorem~\ref{thm:Sturm_Sf_formula-SL}  for a precise statement).

\noindent
\paragraph{The  Morse Index Theorem on (non-compact) graphs.} The Morse index of a critical point of the action functional is defined as the maximal dimension of a subspace on which the second variation is negative definite.  We briefly describe the main result and we refer the reader to Section~\ref{sec:Morse-index-thm} for further details. 

Let $A\in \Cl(L^2(\Gr))$ and we assume that the associated quadratic form $q_A$ is lower bounded. By following the standard Friedrich's extension construction arguments we get the quadratic form  $t_A$ which is the closure of $q_A$ with respect to the Friedrich's inner product.  

Let $\Lambda$ be  a Lagrangian  subspace of $(U,\rho)$ and we let $L_\Lambda$ be the self-adjoint extension of the Sturm-Liouville operator $L$ and we denote by $t_L$ and by $t_{L_\Lambda}$ be the quadratic forms constructed as before by setting $A=L$ or $A=L_\Lambda$. 

\begin{mainthm}\label{thm:triple-index-Morse-Friedrich-extensions-intro}
	Under the above notation, the following equality holds
	\[
	\iMor( L _\Lambda)-\iMor( L_{\Lambda_D})= \itriple(p(\ker  L ^*), \Lambda,\Lambda_D).
	\]
	where $\itriple$ denotes the triple index. (Cfr. Appendix~\ref{sec:Maslov} and references therein).
\end{mainthm}

% Let $x$ be a critical point of a Lagrangian action functional on a quantum graph $\Gr$ with Dirichlet conditions at vertices. Then the Morse index is given by
% \[
% \iMor(L_{\Gr, \Lambda_D}) = \sum_{j=1}^m \sum_{\sigma_j \in (a_j, b_j)} \dim \ker L_{j,\sigma_j,\Lambda_D},\]
% where $L_{j,\sigma_j}$ is the $j$-th edge-restricted Sturm--Liouville operator truncated at $\sigma_j$.

% , explores how the abstract results from Section 3 specialize in explicit cases, particularly star graphs and their unions. These graphs offer computational tractability and are representative of configurations in physical systems (e.g., networks of waveguides).

% A star graph $S_{m+1}$ has a single central vertex connected to $m$ leaves. We assume Dirichlet boundary conditions at the leaves and Kirchhoff-type conditions at the central vertex.

\subsection*{The Morse index for Star Graphs}

For a compact star graph $\GStar{m+1}$, the difference between the Morse index and the Cappell--Lee--Miller index is governed entirely by the branching at the central vertex. More precisely, if $x$ is a non-degenerate critical point of the action functional on $\GStar{m+1}$ with $m$ leaves, then one finds that each leaf contributes the same topological correction to the index relation. This yields the explicit formula
\[
\iMor(x, \GStar{m+1}) - \iCLM(x, \GStar{m+1}) = -(m-1)\, d,
\]
where $d$ is the dimension of the configuration space on each edge.

This identity makes clear that, in the star-graph case, the discrepancy between the two indices is not a dynamical effect but a topological one, determined solely by the number of leaves and the transverse dimension.

This formula links the graph topology directly to the difference between Morse and Maslov indices.

\begin{figure}[h!]
\centering
\begin{tikzpicture}[scale=1.2]
\tikzset{every node/.style={circle, draw=black, inner sep=1pt}}
\node (c) at (0,0) {C};
\node (v1) at (2,0) {1};
\node (v2) at (1.5,1.5) {2};
\node (v3) at (0,2) {3};
\node (v4) at (-1.5,1.5) {4};
\node (v5) at (-2,0) {5};
\foreach \v in {v1,v2,v3,v4,v5}
  \draw[->] (c) -- (\v);
\end{tikzpicture}
\caption{A star graph with 5 leaves and outward orientation.}
\label{fig:star}
\end{figure}
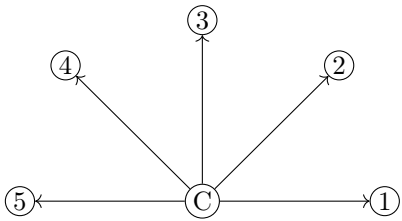

\subsection*{Two-Star Graphs}
Let $\Gr$ be the graph obtained from the two star graphs $\GStar{m_A+1}^A$ and $\GStar{m_B+1}^B$ by gluing them together at their central vertices.

Let $x$ be a critical point of a Lagrangian action on $G$. Then
\[
\iMor(x, \Gr) - \iCLM(x, \Gr) = -(m_A + m_B - 1)d.
\]

\begin{ex}
Let $\Gr = S_2 \cup S_2$ be two single-leaf star graphs joined at the centers, forming a segment. Then:
\[\iMor(x, \Gr) - \iCLM(x, \Gr) = -d,\]
recovering the classical formula for a Sturm--Liouville problem on $[0,1]$ with Dirichlet conditions.
\end{ex}

The paper is organized as follows: 

\tableofcontents

%%%%%%%%%%%%%%%%%%%%%%%%%%%%%%%%%%%%%%%%%%%%%%%%%%%%%%%%%%%%%%%%%
%%
%%
%%
%%
%%
%%
%%
%%%%%%%%%%%%%%%%%%%%%%%%%%%%%%%%%%%%%%%%%%%%%%%%%%%%%%%%%%%%%%%%%

\subsubsection*{Acknowledgement}

The second named author is  deeply grateful to the New York University Abu Dhabi,  for the opportunity to pursue his research in such an inspiring and supportive environment. I also wish to thank the administrative and technical staff of the Institute, for providing excellent working conditions.   Access to the Institute’s library, seminar series, and computing facilities (including the high‑performance Mathematical workstations) provided the essential resources without which this research would not have been possible.

\section{Variational framework on graphs}\label{sec:variational-graphs}

Although we shall not need deep results from graph theory, the notion of graph is central to this paper: we refer the reader to \cite{BR12, BK13, Bol98} for a modern account on the subject.

Throughout the paper it will be assumed that a {\sc graph} is directed and connected, not necessarily simple. That is, we allow for multiple edges adjacent to the same pair of vertices and loops (i.e. edges starting and ending at the same vertex) are also allowed. More precisely, the central objects of the paper are {\sc metric graphs}, i.e. (connected) graphs \(\mathcal{G}=(V, E)\) where each edge  \(e \in E\) is associated with either a closed bounded interval \(I_{e}=\left[a_e, b_{e}\right]\) of length \(\ell_{e}>0\), or a closed half-line \(I_{e}=[a_e,+\infty)\), letting \(\ell_{e}=+\infty\) in this case or a closed half-line \(I_{e}=(-\infty, b_e]\), letting \(\ell_{e}=+\infty\) in this case.  Two edges \(e, f \in E\) joining the same pair of vertices, if present, are distinct objects in all respects: in particular, the corresponding intervals \(I_{e}\) and \(I_{f}\) need not have the same length, and must be considered distinct even in the case where \(\ell_{e}=\ell_{f}\). 

For every \(e \in E\) joining two vertices \(v_1, v_2 \in V\), a coordinate \(t_{e}\) is chosen along \(I_{e}\), in such a way that \(v_1\) corresponds to \(t_{e}=a_e\) and \(v_2\) to \(t_{e}=b_{e}\), or viceversa: if \(\ell_{e}=+\infty\), however, we always assume that the half-line \(I_{e}\) is attached to the remaining part of the graph at \(t_{e}=0\), and the vertex of the graph corresponding to \(t_{e}=+\infty\) is called a {\sc vertex at infinity}. The subset of \(V\) consisting of all vertices at infinity will be denoted by \(V_{\infty}\). With this respect, we shall always assume that
\[
\textrm{ all vertices at infinity of } \mathcal{G} \textrm{ (if any) have degree one }
\]
where, as usual, the degree of a vertex \(v\in \mathrm{V}\) is the number of edges incident at $v$ (counting twice any self-loop at $v$, of course).  Finally, we assume that the cardinality  of \(|E|=m\) as well as the one of   \(|V|=n\) is  finite.

A connected metric graph \(\mathcal{G}\) has the natural structure of a locally compact metric space, the metric being given by the shortest distance measured along the edges of the graph. Observe that
\[
\mathcal{G} \text { is compact } \Longleftrightarrow \text { no arc  of } \mathcal{G} \text { is a half-line } \Longleftrightarrow V_{\infty}=\emptyset .
\]
\begin{note}
With some abuse of notation, we often identify an arc  \(e\) with the corresponding interval \(I_{e}\) 
\end{note}
\begin{note}
\label{notation:graph}
Sometimes it will be convenient to use \emph{ad-hoc} notation for directed graphs.
A directed graph is the given of two sets, \(V\) the set of vertices, and \(E\) the
set of edges. Adjacency rules are determined by a function \(\zeta\colon E\to V\times V\); we say that \(v_1\in V\) is adjacent to \(v_2\in V\) if one \((v_1,v_2)\in\zeta(E)\)
or \((v_2,v_1)\in\zeta(E)\). Since we are dealing with \emph{directed graphs}, it makes sense to map an edge to an ordered pair. Since we are not requiring graphs to be simple, we allow \(\zeta\)
to be non-injective (allowing the existence of multiple edges adjacent to the same pair of vertices) and to have element of the diagonals in its image (allowing loops).
\end{note}
Topologically, the metric space \(\mathcal{G}\) is the disjoint union \(\bigsqcup I_{e}\) of its edges, with some of their endpoints glued together into a single point (corresponding to a vertex \(v \in V \backslash V_{\infty}\) ), according to the topology of the graph \(\mathcal{G}\) (using the same symbol \(\mathcal{G}\) for both the metric graph and the induced metric space should cause no confusion). We point out that any vertex at infinity \(\mathrm{V} \in V_{\infty}\) is of course a vertex of the graph \(\mathcal{G}\), but is not a point of the metric space \(\mathcal{G}\).

With \(\mathcal{G}\) as above, a function \(x: \mathcal{G} \rightarrow \mathbb{R}\) can be regarded as a family of functions \(\left(x_{j}\right)_{j=1}^m\), where \(x_{j}: I_j \rightarrow \mathbb{R}\) is the restriction of \(x\) to the arc  \(I_{j}=(a_j, b_j)\). Endowing each arc  \(I_{j}\) with Lebesgue measure, one can define \(L^{2}\) spaces over \(\mathcal{G}\) in the natural way, with norm
\[
\|x\|_{L^{2}(\mathcal{G})}^{2}=\sum_{j=1}^m \left\|x_{j}\right\|_{L^{2}\left(I_{j}\right)}^{2}, \quad x=\left(x_{j}\right)_{j=1}^m.
\]
So, $L^2(\Gr)\=\bigoplus_{j=1}^m L^2(I_j;\mathbb{R}^d)$.
Similarly, the Sobolev space \(W^{1,2}(\mathcal{G})\) is defined as the set of those functions \(x: \mathcal{G} \rightarrow \mathbb{R}\) such that
\[
x=\left(x_{j}\right)_{j=1}^m \textrm{ is continuous on } \mathcal{G}\textrm{ and } x_{j} \in W^{1,2}\left(I_{j};\mathbb{R}^d\right) \textrm{ for every } j=1, \ldots,m
\]
with the norm
\[
\|x\|_{W^{1,2}(\mathcal{G})}^{2}=\int_{\mathcal{G}}\left(\left|x^{\prime}(t)\right|^{2}+|x(t)|^{2}\right) d t=\sum_{j=1}^m \int_{I_{j}}\left(\left|x_{j}^{\prime}\left(t_{j}\right)\right|^{2}+\left|x_{j}\left(t_{j}\right)\right|^{2}\right)\, d t_{j}.
\]
Note that \(W^{1,2}(\mathcal{G})\) can be identified with a closed subspace (determined by the continuity of at the vertices of \(\mathcal{G}\)) of the direct sum  \(\bigoplus_{j=1}^m W^{1,2}\left(I_{j}\right)\). In terms of the coordinate system \((x_{j})_{j=1}^m\), continuity on 
\(\mathcal{G}\) means that, whenever two edges \(e, f\) meet at a vertex $v$ of \(\mathcal{G}\), the corresponding branches of \(x\) satisfy a no-jump condition. Notice that,  vertices at infinity are never involved in these continuity conditions: on the other hand, if \(x\in H^{1}(\mathcal{G})\), then automatically
\[
I_{j}=[a_j,\infty) \quad \Rightarrow \quad \lim _{t_{j} \rightarrow\infty} x_{j}\left(t_{j}\right)=0,
\]
because \(x_{j} \in W^{1,2}\left(I_{j}\right)\).  $W^{2,2}(\Gr, \R^d)$ is a collection of functions $x \in W^{1,2}(\Gr, \R^d)$ such that $x_j \in W^{2,2}(I_j, \R^d)$ for every $j=1, \ldots, m$. 

Given a collection of linear operators $L_j: W^{2,2}(I_j, \R^d) \to L^2(I_j, \R^d)$ parametrized by $j=1, \ldots, m$, we define  the operator $L: W^{2,2}(\Gr, \R^d) \to L^2(\Gr, \R^d)$ as a collection  of $L_j:W^{2,2}(I_j, \R^d) \to L^2(I_j, \R^d) $ such that 
\[
(L x)_j\= L_j  x_j \qquad j=1, \ldots, m.
\]
\begin{rem}
    We would like to remark that $L x$ is not a function because each of the $x_j$ are defined on different domains for every $1 \le j \le m$.  
\end{rem}
A {\sc Lagrangian function} $\mathscr L$ on the graph $\Gr$ is a map
such that 
\begin{description}
    \item[(i)] $\mathscr L\in \mathscr C^0(\Gr\times T\R^d,\R)$ 
    \item[(ii)] $\mathscr L_j :=\mathscr L|_{I_j} \in \mathscr C^2(I_j\times T\R^d,\R)$ for every $j=1, \ldots, m$.
\end{description}
In what follows we always assume that each $ \mathscr L_j$ satisfies the {\em Legendrian convexity condition}, namely
\begin{description}
\item[(L1)] $ L_j$ is $\mathscr C^2$-strictly convex on the fibers of $T\R^d$, that is $\partial_{vv}^2 L>0$
\item[(L2)] $ L$ is exactly quadratic in the velocities.
\end{description}
Let $A_\Gr: W^{1,2}(\Gr) \longrightarrow \R$ be the {\sc Lagrangian action functional on $\Gr$} defined by
\begin{equation}\label{eq:action-functional}
A_\Gr(x) =\int_\Gr \mathscr L\big(t, x(t), x'(t)\big)\, dt\= \sum_{j=1}^m\int_{I_j} \mathscr L_j\big(t, x_j(t), x'_j(t)\big)\, dt
\end{equation}
and let $M \subset \R^{N_a} \times \R^{N_b}$ be an affine $k$-dimensional subspace, where $k=0, \ldots, N_a+N_b$  being  $N_a\=m_a\, d$ (resp. $N_b\=m_b\, d$) and where $m_a$ (resp. $m_b$) is the total number of finite left (resp. right) endpoints. 

We set $\ev_c: W^{2,2}(\Gr, \R^d)\to \R$ be the {\bf evaluation  map at $c$} for $c=a$ or $c=b$,  defined by 
\[
\ev_a(x)\=(x_{i_1}(a_{i_1}), \ldots, x_{i_{m_a}}(a_{i_{m_a}})) \in \R^{N_a}\textrm{ and } 
\ev_b(x)\=(x_{j_1}(b_{j_1}), \ldots, x_{j_{m_b}}(b_{j_{m_b}})) \in \R^{N_b}
\]
where 
\begin{itemize}
\item $\Set{i_1, \ldots, i_{m_a}}$ is the set of indices of all  intervals having finite left endpoints
\item $\Set{j_1, \ldots, j_{m_b}}$ is the set of  indices of all  intervals having finite right endpoints 
\end{itemize}
%Denoting the elements of the tangent bundle $T\R^d$ by $(q,v)$, 
We denote $(q,v) $ the elements of the tangent bundle $T\R^d$.
The solutions of the Euler-Lagrange equation  with the boundary condition \((\ev_a(x), \ev_b(x))\in M\) are 
%to the Lagrangian \(\mathscr L\) is 
the critical points of \(A_\Gr\) in the Hilbert manifold
\[
W^{1,2}_{M}(\Gr)=\Set{x \in W^{1,2}\left(\Gr\right)| \big(x(a), x(b)\big) \in M} .
\]
An {\sc extremal curve satisfying  the boundary condition \(M\)} is a solution of the  following boundary value problem
\begin{equation}\label{eq:Euler-Lagrange}
\begin{cases}
\dfrac{d}{d t} \dfrac{\partial \mathscr L}{\partial v}(t, x, x')-\dfrac{\partial \mathscr L}{\partial q}(t, x, x')=0 \qquad t \in \Gr \\[7pt]
\big(x(a), x(b)\big) \in M \\[7pt] \left(\ev_a\left[\dfrac{\partial \mathscr L}{\partial v}\big(t, x, x'\big)\right],-\ev_b\left[\dfrac{\partial \mathscr L}{\partial v}\big(t, x, x'\big)\right]\right) \in \Big(T_{(x(a), x(b))} M\Big)^\perp
\end{cases}
\end{equation}
where \((T_{(x(a), x(b))} M)^\perp\) is the orthogonal complement of \(T_{(x(a),x(b))} M\) in \( \R^{N_a} \times \R^{N_b}\). By using the Legendrian transformation defined by \(y=\frac{\partial\mathscr  L}{\partial v}(t, x, x')\) and \(H(t, y, x)=y \cdot x'-\mathscr L(t, x, x')\), the Equation\eqref{eq:Euler-Lagrange} can be converted into the Hamiltonian  equation 
\[
z'(t)=J \nabla  H(t, z(t)) \qquad t \in \Gr
\]
with \(z(t)=(y(t), x(t))=\left(\dfrac{\partial \mathscr L}{\partial v}(t, x(t), x'(t)), x(t)\right)\) where $J= \begin{pmatrix}
    0 & -\Id_d\\
    \Id_d &0
\end{pmatrix}$ denotes the standard complex structure of $(T^*\R^d, \Omega)$, where $\Omega$ denotes the standard symplectic  structure.  We set 
\begin{itemize}
\item $T^*\R^{N_a}:= (T^*\R^{d})^{m_a}$ and $\Omega_a=\oplus_{j=1}^{m_a} \Omega$
\item $T^*\R^{N_b}:= (T^*\R^{d})^{m_b}$ and $\Omega_b=\oplus_{j=1}^{m_b} \Omega$ 
\end{itemize}
Then the boundary conditions of the corresponding Hamiltonian  equation can be expressed by  a Lagrangian subspace of the symplectic space \(\left(T^* \R^{N_a} \oplus T^* \R^{N_b},-\omega_a \oplus \omega_b\right)\). 

Let $W\=T_{\big(\ev_a(x), \ev_b(x)\big)}M$ and we observe that  \(W, W^{\perp} \subset \R ^{N_a} \oplus \R ^{N_b}\). Let \(\mathcal J:=-J  \oplus J\) be the block diagonal matrix representing the complex structure of   \(T^*\R^{N_a} \oplus T^*\R ^{N_b}\).  By using the canonical embedding 
\[
\R^{N_a}\oplus \R^{N_b} \longrightarrow T^* \R^{N_a} \oplus T^*\R^{N_b}
\]
we identify the subspaces $W $ and $ W^\perp$ with $\overline  W$ and  $\overline W^\perp$, respectively.  Then we set 
\[
\Lambda_W=\mathcal J(\overline W^{\perp}) \oplus \overline  W.
\]
It is straightforward to check that \(\operatorname{dim}(\Lambda_W)= N_a+N_b\) and \(\left.(-\Omega_a \oplus \Omega_b)\right|_{\Lambda_W}=0\), so \(\Lambda_W\) is a Lagrangian subspace of \(\left(T^*\R^{ N_a} \oplus T^*\R^{N_b},-\Omega_a \oplus \Omega_b\right)\) and the boundary conditions at Equation~\eqref{eq:Euler-Lagrange} correspond to 
\[
(\ev_a(z), \ev_b(z)) \in \Lambda_W.
\]
Suppose that \(x\) is an extremal curve of \(A_\Gr\) in \(W^{1,2}_W(\Gr)\). We let 
\begin{multline}
P_j(t)=\dfrac{\partial^{2} \mathscr L_j}{\partial v^{2}}(t, x_j(t), x'_j(t)) \qquad  Q_j(t)=\dfrac{\partial^{2} \mathscr L}{\partial q \partial v}(t, x_j(t), x'_j(t))\\
\qquad \textrm{ and }\qquad  R_j(t)=\dfrac{\partial^{2} \mathscr L}{\partial q^{2}}(t, x_j(t), x'_j(t))\qquad t \in I_j\quad  \textrm{ and } \quad j=1, \ldots, m
\end{multline}
and we denote by $P, Q$ and $R$ the real-valued functions defined on $\Gr$ as follows
\[
P|_{I_j}(t)\=P_j(t) \qquad  Q|_{I_j}(t)\=Q_j(t) \qquad  R|_{I_j}\=R_j(t)\quad  j=1, \ldots, m.
\]
Then the {\sc Index form} at \(x\) is the quadratic form on $W^{1,2}_W(\Gr)$ given by
\begin{multline}
\mathcal{I}(\xi, \eta)=\int_\Gr  \Big[\langle P(t)\xi'(t)+Q(t)\xi(t),  \eta'(t)\rangle +\langle \trasp{Q}(t)\,\xi'(t),  \eta(t)\rangle +\langle R(t) \xi(t),   \eta(t)\rangle\Big]\, dt
\\[7pt] \=\sum_{j=1}^m\int_{a_j}^{b_j}\Big[\langle P_j(t)\xi_j'(t)+Q_j(t)\xi_j(t),  \eta'_j(t)\rangle +\langle \trasp{Q_j}(t)\xi'_j(t),  \eta_j(t)\rangle +\langle R_j(t) \xi_j(t),   \eta_j(t)\rangle\Big]\, dt
\end{multline}
for every \(\xi,\eta\in W^{1,2} (\mathcal{G})\).

By linearizing the Euler-Lagrange equation  at \(x=(x_1, \ldots, x_m)\), then  we get the Sturm-Liouville boundary value problem 
\begin{equation}\label{eq:3-9}
\begin{cases}
-\dfrac{d}{d t}\Big[P(t) y'(t)+Q(t) y(t)\Big]+\trasp{Q}(t) y'(t)+R(t)\, y=0 \qquad t \in \Gr\\[7pt]
\big(\ev_a(y), \ev_b(y)\big) \in W\\[7pt] \left(\ev_a\left[\dfrac{\partial \mathscr L}{\partial v}\big(t, y, y'\big)\right],-\ev_b\left[\dfrac{\partial \mathscr L}{\partial v}\big(t, y, y'\big)\right]\right) \in W^\perp
\end{cases}
\end{equation}
We observe that  \(y\) solves the boundary value problem given at Equation~\eqref{eq:3-9} if and only if \(y \in \operatorname{ker}(\mathcal{I})\). Moreover, by setting  \(z \equiv z(t)=(P(t) y'(t)+Q(t) y(t), y(t))\), we get that Equation~\eqref{eq:3-9} reduces to the  linear Hamiltonian  system
\begin{multline}\label{eq:Hamiltonian -system}
z'(t)=J B(t) z(t) \qquad t \in \Gr\qquad 
\textrm{ where }\quad 
B(t)=\begin{bmatrix}
P^{-1}(t) & -P^{-1}(t) Q(t) \\
-Q^{T} P^{-1}(t) & Q^{T}(t) P^{-1}(t) Q(t)-R(t)
\end{bmatrix}.
\end{multline}

%%%%%%%%%%%%%%%%%%%%%%%%%%%%%%%%%%%%%%%%%%%%%%%%%%%%%%%%%%%%%%%%%
%%
%%
%%
%%
%%
%%
%%
%%%%%%%%%%%%%%%%%%%%%%%%%%%%%%%%%%%%%%%%%%%%%%%%%%%%%%%%%%%%%%%%%

\section{Gelfand-Robbin quotient and Cauchy data spaces}\label{sec:sf-Gelfad-Robbin}

In this section we are going to introduce the used notation and the basic definitions about {\bf   factor space of abstract boundary value} and  {\bf   Cauchy data space} that will be used throughout the whole manuscript. Our basics references are  \cite{BZ18, Fur00, Fur04, HPW25, HPW25} and references therein.

Let $(H, \langle \cdot, \cdot \rangle, \omega)$ be a (real and separable) Hilbert space with an inner product $\langle \cdot, \cdot \rangle$ and associated norm $\|\cdot\|$ and we assume that  $H$ has a symplectic form $\omega(\cdot, \cdot)$ that is a bounded, nondegenerate, skew-symmetric bilinear form.   Once the symplectic form $\omega$  has been given, up to eventually replacing  the inner product by an equivalent one, we may assume that there exists an orthogonal transformation $J: H \to H$ such that 
\begin{equation}\label{eq:1}
\omega(x,y)=\langle J\,x, y\rangle \qquad x,y \in H \quad \textrm{ and } \quad J^2=-\Id. 
\end{equation}
We denote by $\Grass(H)$ the set of all closed linear subspaces of $H$ and we set 
\[
\dist (u,V)\=\inf_{v \in V}\|u-v\|.
\]
 Given  $U,V \in \Grass(H)$, we denote by $S_U$ the unit sphere of $U$ and we set 
\begin{equation}
	\delta(U,V):=\begin{cases}
	\sup_{u \in S_U} \dist(u, V) & \textrm{ if } U \neq (0)\\[3pt]
	0 &  \textrm{ if } U = (0)
\end{cases}\quad \textrm{ and } \quad 
\widehat \delta(U,V):= \max\{\delta(U,V), \delta(V,U)\}.
\end{equation}
$\widehat \delta(U,V)$ is called the {\bf  gap} between $U$ and $V$.
 $U \in \Grass(H )$, there exists a unique orthogonal projection $P_U$ onto $U$ which is a bounded operator on $H$. Given $U, V \in \Grass(H)$, we consider the norm of the difference of the corresponding orthogonal projectors:
\begin{equation}\label{eq:DG}
d_G(U,V):= \norm{P_U- P_V}.
\end{equation}
It can be shown  that metric topology induced by $d_G$ is equivalent to the gap topology. 
\begin{note}
We denote by $\Cl(H)$ (resp. $\Cl^s(H)$) be the set of all {\bf  closed} (resp. {\bf  closed and symmetric}) and densely defined   operators\footnote{
Since  the adjoint of a densely defined linear operator is closed then  selfadjoint operators on $H$ are contained in $\Cl(H)$.} and we denote by $\Grn{\#}$ denotes the graph of the operator $\#$.
\end{note}
We observe that the gap metric induces in a natural way a metric on $\Cl(H)$ (resp. $\Cl^s(H)$) given by 
\[
d(T, S):=d(\Grn{T} , \Grn{S}) \qquad S,T \in \Cl(H)
\]
where $\Grn{\#}$ denotes the graph of the operator $\#$.  

Let $A \in \Cl^s(H)$,  $A^*\in \Cl^s(H)$ be its adjoint and we denote by  $\dom(A)$ and $\dom(A^*)$  the domain of $A$ and $A^*$, respectively. We  define the {\bf   graph inner product on $\dom(A^*)$}  as follows
\[
\langle x,y \rangle^G:=\langle x, y \rangle + \langle A^* x, A^* y\rangle \qquad x, y \in \dom(A^*)
\]
and we observe that  $\dom(A^*)$ equipped with the graph norm  becomes a Hilbert space  and $\dom(A)$ is a closed subspace. (Cfr. \cite[Example 2.2]{Fur04} for further details).  We define the (Hilbert) {\bf factor space}  or the space of  {\bf abstract boundary value} or {\bf Gelfand-Robbin quotient} as the quotient
\[
\beta(A):= \dom(A^*)/\dom(A).
\]
We equip $\beta(A)$ by the  the nondegenerate skew-symmetric bilinear form $\omega$  given by 
\begin{equation}\label{eq:symplectic-form}
\omega\big(\gamma(x), \gamma(y)\big)=  \langle x, A^* y\rangle -\langle A^*\, x, y\rangle
\end{equation}
where $\gamma:\dom(A^*) \to \beta(A)$ is the canonical projection onto the quotient defined by $\gamma(x):=[x]:= x +\dom(A)$. We refer to $\gamma$ as the {\bf factor space projection}. For $\mu \in \Grass(H)$, we  denote by $\mu^\circ$ the {\bf   $\omega$-annihilator} of $\mu$, i.e.
\[
\mu^\circ \=\Set{x \in H| \omega(x,y)=0 \textrm{ for all } y \in \mu}
\]
and we denote the orthogonal complement of $\mu$ by $\mu^\perp$. We observe that for any subspace $\mu$ the annihilator $\mu^\circ$ is closed and by the nondegeneracy of the symplectic form we also get that the map $\mu \mapsto \mu^\circ$ is idempotent. As in the finite dimensional case, the following holds.  
\begin{defn}
Under the above notation, let $\mu\in \Grass(H)$. Then the subspace $\mu$ is termed
\begin{multicols}{2}
\begin{itemize}
	\item {\bf Isotropic} if $\mu \subset \mu^\circ$
	\item  {\bf Lagrangian} if $\mu= \mu^\circ$
	\item {\bf Coisotropic} if $\mu^\circ \subset \mu$
	\item {\bf Symplectic}  if  $ \mu \oplus \mu^\circ=H$.
\end{itemize}	
\end{multicols}
\end{defn}
\begin{rem}
By taking into account the  compatibility assumption between the symplectic form, the inner product and the almost complex structure $J$, we get the following fact
\begin{itemize}
	\item If $\mu$ is Lagrangian, then $\mu$ is a closed subspace, $J\, \mu$ is also Lagrangian and $J\, \mu=\mu^\perp$. Conversely, if $\mu$ is a closed subspace and $\mu^\perp = J \, \mu$, then $\mu$ is a Lagrangian subspace.
\end{itemize}
\end{rem}
\noindent
Let $\Lag(H, \omega)$ or just $\Lag(H)$  denote the {\bf   Lagrangian Grassmannian} of the  symplectic Hilbert space $H$ space namely the set of all Lagrangian subspaces of $H$. It is well-known  that $\Lag(H)$ is an infinite dimensional differentiable manifold modeled on the Banach space of bounded and selfadjoint operators. (Cfr. \cite[Corollary 2.25]{Fur04} and references therein).  So, from now on, we regard $(\Lag(H), d_G)$ as a metric space where $d_G$ is the metric defined at Equation~\eqref{eq:DG}.

%%%%%%%%%%%%%%%%%%%%%%%%%%%%%%%%%%%%%%%%%%%%%%%%%%%%%%%%%%%%%%%%%
%%
%%
%%
%%
%%
%%
%%
%%%%%%%%%%%%%%%%%%%%%%%%%%%%%%%%%%%%%%%%%%%%%%%%%%%%%%%%%%%%%%%%%

\subsection{A spectral flow formula in the Cauchy data spaces}

Let $A, A^*\in \Cl^s(H)$ with domains respectively given by $\dom(A), \dom(A^*)$, and we consider the  factor space 
\[
\beta(A)= \dom(A^*)/\dom(A).
\]
Let $\gamma: \dom(A^*) \to \beta(A) $ be  quotient map  and we observe that the subspace $\gamma (\ker A^*)$ defined by 
\[
\gamma (\ker A^*)=\Set{\gamma(x)| x \in \ker A^*}= \big(\ker A^*+ \dom(A) \big)/\dom(A). 
\]
is an {\em isotropic subspace} of the symplectic space $\beta$. We refer to $\gamma (\ker A^*)$ as the {\bf Cauchy data space of the operator $A$.}
The next well-known result gives a characterization of the selfadjointness  of an operator $A$ on $\dom$ in terms of the symplectic properties of $\gamma(\dom)$. 
\begin{lem}\label{lem:abstract_fundamental_solution}
Let $A\in \Cl^s(H)$ and let $\dom $ be denote a subspace such that $\dom(A)\subset\dom \subset \dom(A^*)$. Then the restriction of the adjoint operator  $A^*$ to the subspace $\dom$ is selfadjoint, if and only if, the subspace $\gamma(\dom)$ is a Lagrangian subspace in $\beta(A)$.	
\end{lem}
\begin{proof}
We refer the interested reader to \cite[Proposition 6.1]{Fur04} and references therein.
\end{proof}
\begin{note}
We denote by $\CFs(H)$ the set of all closed densely defined and symmetric Fredholm   operators on $H$. 
\end{note}
\begin{lem}\label{thm:beta-finit-dimensional}
Let $A\in \CFs(H)$. Then 
\begin{itemize}
\item[(a)] $\beta(A)$ is a finite dimensional symplectic vector space
	\item[(b)] $A$ has at least one selfadjoint Fredholm  extension, that is there exists a subspace $ \dom $ (closed in the graph norm topology) such that $A_{\dom}:= A^*\big \vert_{\dom}$ is selfadjoint and Fredholm .
	\end{itemize}
\end{lem}
\begin{proof}
We refer the interested reader to \cite{HPW25} for the proof. 
 \end{proof}
Ordinary differential operators we are interested in  are closed Fredholm  and selfadjoint and as a consequence of Lemma~\ref{thm:beta-finit-dimensional}, the involved factor space are finite dimensional. 

We introduce the following {\bf Unique Continuation Property}
\begin{itemize}
	\item {\bf (H1) } Let $A \in \Cl^s(H)$  and we assume that $\ker A^* \cap \dom(A)=\{0\}$. 
\end{itemize}
 \begin{rem}
For ordinary differential operators this is always satisfied. However, it is well-known  that this property is not necessarily true for a general elliptic partial differential operator and it is related to the Carleman, Fefferman-Phong estimates etc. 
 \end{rem}
 The first basic consequence of (H1)  is the following result. 
\begin{lem}
	Let $A \in \CFs(H)$ and we assume that (H1)  holds. Then 
	\[
	\gamma\big(\ker A^*\big)
	\]
	is Lagrangian.
\end{lem}
\begin{proof}
	We refer the interested reader  to \cite[Proposition 6.3 (a)]{Fur04}.
\end{proof}

Let $I \ni s \longmapsto A_s:= A+s\Id \in \CFs(H)$ be a path of closed symmetric and Fredholm  operators  and we observe  that 
\begin{equation}\label{eq:dominio-fisso}
\dom(A_{s}) =\dom(A)  \quad \textrm{ and } \quad \dom(A^*_{s}) =\dom(A^*) \qquad s \in I.
\end{equation} 
(Cfr. \cite[Chapter 4, Section 1.1]{Kat80} for further details). 
 By  Equation~\eqref{eq:dominio-fisso}, we immediately get that both the factor space $\beta(A_s)$ and the associated form $\omega_s$ induced by $A_s$  are  independent on $s$ and so it  coincides with the factor space of $(\beta(A), \omega)$, where   $\omega$ is given by 
 \[
 \omega(\gamma(x), \gamma(y))=\langle x, A^* y\rangle- \langle A^* x, y\rangle \qquad \textrm{ for all } \quad x, y \in \dom(A^*).
 \]
\begin{lem}
We let $V_s:= \ker A^*_{s}$ and we  assume that condition (H1)  holds.    Then the map  $s \mapsto\gamma(V_s)$ is gap continuous  in $\Lag(\beta, \omega)$. 
\end{lem}
\begin{proof}
We refer the interested reader to \cite{HPW25} for the proof. 
	 \end{proof}
The following spectral flow formula for an operator pencil in $\CFs(H)$ holds. 
\begin{prop}\label{thm:main1-abstract}
For $s\in I$, we let   $A_s=A+s\Id \in \CFs(H)$  and we assume that (H1)   holds.  Given $L \in \Lag(\beta(A), \omega)$, let us   consider the selfadjoint extension $A_L$.  Then we get 
\begin{equation}
 \spfl(A_L+sI; s \in I)=-\iCLM(L,\gamma(V_s); s \in I)
 \end{equation}
 where $\spfl$ and $\iCLM$ denote the {\sc spectral flow} and the {\sc Maslov index}, respectively. (Cfr. Appendix~\ref{sec:spectral-flow-Maslov-index}, for the definition and the basic properties).
\end{prop}
\begin{proof}
We refer the interested reader to \cite{HPW25} for the proof. 
\end{proof}

%%%%%%%%%%%%%%%%%%%%%%%%%%%%%%%%%%%%%%%%%%%%%%%%%%%%%%%%%%%%%%%%%
%%
%%
%%
%%
%%
%%
%%
%%%%%%%%%%%%%%%%%%%%%%%%%%%%%%%%%%%%%%%%%%%%%%%%%%%%%%%%%%%%%%%%%

\section{The Morse index theorem  on (non)compact graphs}\label{sec:Morse-index-thm}

This section contains the core of the paper, the Morse index theorem 
in non-compact quantum graphs. An important ingredient to achieve this result is a spectral flow formula for  gap continuous paths of selfadjoint Fredholm  operators recently proved by authors in \cite{HPW25}.

%%%%%%%%%%%%%%%%%%%%%%%%%%%%%%%%%%%%%%%%%%%%%%%%%%%%%%%%%%%%%%%%%
%%
%%
%%
%%
%%
%%
%%
%%%%%%%%%%%%%%%%%%%%%%%%%%%%%%%%%%%%%%%%%%%%%%%%%%%%%%%%%%%%%%%%%

Let $a,b \in \overline\R$, $I:=[0,1]$ and we denote by $\Sym_d(\R)$ the set of $(d\times d)$-symmetric matrices. For each $s \in [0,1]$, we define the differential expression $l_s$ as follows
\begin{equation}\label{eq:Sturm-Liouville-operator-Gr}
		\ell_s:=-\dfrac{d}{dt}\left(P(t)\dfrac{d}{dt}+ Q(t)\right)+ \trasp{Q}(t)\dfrac{d}{dt}+ R(t) + C_s(t)\qquad \textrm{ for } \quad (t,s)\in \Gr\times I 
\end{equation}
meaning that the restriction onto each arc  $I_j$ of the graph, i.e.  $l_s|_{I_j}$  is given by 
\begin{equation}\label{eq:Sturm-Liouville-operator-s}
		\ell_{j,s}:=-\dfrac{d}{dt}\left(P_j(t)\dfrac{d}{dt}+ Q_j(t)\right)+ \trasp{Q}_j(t)\dfrac{d}{dt}+ R_j(t) + C_{j,s}(t)\qquad \textrm{ for } \quad (t,s)\in I_j\times I 
\end{equation}
where\footnote{Here the regularity assumptions on the coefficients are meant in the sense that the  restriction on each edge has the prescribed regularity and are consequences of the assumptions on the Lagrangian.}
\begin{multline}\label{eq:assumptions-coeff}
		P^{-1} \in \mathscr C^1\big(\Gr, \Sym_d(\R)\big)  \qquad 
		Q \in\mathscr C^1\big(\Gr, \Mat_d(\R)\big) \qquad R \in \mathscr  C^0\big(\Gr,\Sym_d(\R)\big) \\[3pt] \textrm{ and }
         s\mapsto C_s \in \mathscr C^0(\Gr,\Sym_d(\R))\cap L^\infty\big(\Gr,\Sym_d(\R)\big).
        % s\mapsto C_s \in \mathscr C^0\big(\Gr, \mathscr C^0(\Gr,\Sym_d(\R)\big).
\end{multline}
Under the regularity assumptions provided at Equation~\eqref{eq:assumptions-coeff}, we get that the formal differential expression $\ell_s$ given at Equation~\eqref{eq:Sturm-Liouville-operator-s} defines a   second order  linear differential operator called {\sc Sturm-Liouville operator} or {\sc SL-operator}, for brevity. (Cfr. \cite{Zet05} for further details).

We denote by  $L_s$ (resp. $L_s^*$) the {\bf minimal} (resp. {\bf maximal}) {\bf operator associated to $\ell_s$} whose domains are the following: 
\begin{multline}
\dom( L_s ^*)=\bigoplus_{j=1}^m W^{2,2}(I_j, \R^d)\cong W^{2,2}(\Gr, \R^d)  \qquad 
 \dom( L_s )=\bigoplus_{j=1}^m W^{2,2}_0(I_j, \R^d)\cong W_0^{2,2}(\Gr, \R^d)\\[5pt] \qquad \beta= W^{2,2}(\Gr, \R^d)/W^{2,2}_0(\Gr, \R^d) \cong  T^*\R^{N_a}\oplus T^*\R^{N_b}.
 \end{multline}
 We observe that both these domains are  independent on $s$. 
The isomorphism $\phi$ between $\beta$ and $T^*\R^{N_a}\oplus  T^*\R^{N_b}$  is defined by 
 \begin{multline}
 \phi(x)=\big(\ev_a(x^{[1]}), \ev_a(x) , \ev_b(x^{[1]}) , \ev_b (x)\big)\qquad \textrm{ where }\qquad 
   x^{[1]}(t)=P(t)\, x'(t)+ Q(t)\,x(t).
     % $\ev_c(x^{[1]})\=x^{[1]}_1(a_1), \ldots,x^{[1]}_{m_a}(a_{m_a})\quad \textrm { and  } \quad  
  % x^{[1]}(t)\=(x^{[1]}_{1}(t), \ldots,x^{[1]}_{m}(t))\qquad \textrm{ where } \qquad 
  % \phi_{1,a}(x_{1}(a)), \ldots, \phi_{m,a}(x_{m_a}(a)), \phi_{1,b}(x_{1}(b)), \ldots, \phi_{m,b}(x_{m_b}(b))\big) \\\qquad \textrm{ where } \qquad   \phi_{j, a}(x_j(c))=\big(x_j^{[1]}(a^+), x_j^{[1]}(b)\big)
  % \quad \textrm{ for  }\quad  c=a,b \textrm{ and where }
  % \quad \textrm{ with } c_j \in \Set{a_1, \ldots, a_{m_a}, b_1, \ldots, b_{m_b}} \\
  % \textrm { and } \qquad x(a)=\big(x_1(a_1), \ldots, x_{m_a}(a_{m_a})\big) \quad x(b)=\big(x_1(b_1), \ldots, x_{m_b}(b_{m_b})\big).
  % \quad \textrm{ and  } \quad  x^{[1]}_j(b_j)=P_j(b_j)\, x'_j(b)+ Q_j(b_j)\,x_j(b_j).
 \end{multline}
 We refer to the vector $x^{[1]}(c)$ as the {\sc quasi-Wronskian of $x$ at $c$.} 

 Let $f,g \in \dom(L^*)$ and we let 
\[
\omega(f,g)=\langle f, L^* g \rangle- \langle L^* f , g\rangle. 
\]
Integrating by parts, we get that 
\begin{multline}
\omega(f,g)=[f,g](b)-[f,g](a)\qquad \textrm{ where  }\\[3pt]
[f,g](c):= \Omega_c\big(\ev_c(f^{[1]}), \ev_c(f),\ev_c(g^{[1]}), \ev_c(g) \big)\quad \textrm{ for } \quad c=a,b.
\end{multline}
% We define 
% \begin{multline}
% \omega(f,g)= \sum_{j=1}^m \omega_j(f_j, g_j) \qquad \textrm{ where } \qquad f,g:\Gr \to \R \qquad \textrm{ and } \qquad\\
% f_j\=f|_{I_j} \quad \textrm{ and } \quad g_j\=g|_{I_j}.
% \end{multline}
% \begin{rem}
% It	is worth observing that the bracket $[f,g](x)$ coincides with the standard symplectic structure  once $\R^{2N}$ is equipped by  the local system of  coordinates $(P(x)f'(x)+ Q(x)f(x) , f(x))$.
% \end{rem}
% \begin{note}
% We set  
% \[
% [f,g](a):=[f,g](a).
% \]
% \end{note}
We refer to Appendix~\ref{sec:sf-Gelfad-Robbin} for the definition and properties of the Gelfand-Robbin quotients.
We start by introducing the following condition. 
\begin{itemize}
	% \item[{\bf  (H1)}]   $I_j\neq \R$ for each $j=1, \ldots, m$
 %    at least one of the instant $a_i$ or $b_i$ is a regular endpoint for the the family $s\mapsto L_s$ of SL-operators on $\Gr$. This means that, for every $j=1, \ldots, m$, the instant $t=b_j$ is finite and   for every $s \in I$ it is a regular endpoint  for the operator $\ell_{j,s}$ meaning that 
	% the  paths $t \mapsto P'_j(t)$, $t \mapsto Q_j(t)$,  $t \mapsto R_j(t)$ and  $t \mapsto C_{j,s}(t)$  are continuous up to the instant $t=b_j$
     \item[{\bf (H2)}] The path $I \ni s\mapsto L^*_s \in \CFs(L^2(\Gr))$ is  gap-continuous  in the space of closed densely defined symmetric Fredholm  operators on $L^2(\Gr)$.
     % ; namely, for each $j=1, \ldots, m$ the path that $I \ni s\mapsto L^*_{j,s}\= L^*_s|_{I_j} \in \CFs(L^2(I_j))$ is a gap continuous path of symmetric Fredholm  operators on $L^2(I_j)$.
	\end{itemize}
As a direct consequence of the existence and uniqueness theorem for regular linear odes, we get that the unique continuation property holds for the Sturm-Liouville operator $L$. 
\begin{lem} \label{lm:unique_extension}
Let $ L_s $ be  the minimal operator associated to the Sturm-Liouville operator $\ell_s$ defined at Equation~\eqref{eq:Sturm-Liouville-operator-s}. 
Then we get that the equation 
\[
L_s \,x=0
\] 
has no nontrivial solutions in $\dom(L_s)= W_0^{2,2}(\Gr,\R^d)$. 
\end{lem}
\begin{proof} By construction, it is enough to prove the result for a single edge, let's say for the $j$-th. Arguing by contradiction, we assume that there exists a nontrivial solution $0 \neq x$ in $L^2(I_j)$ and since the operator $ L_{j,s} $ is only one-sided singular at $t=a_j$. We note that for each $x_j\in \dom (L_{j,s})$ it holds that $x_j(b_j)= x'_j(b_j)=0$. For every  $d_j\in (a_j,b_j)$ the equation $ L_{j,s} \, x=0$ is a regular Sturm-Liouville equation on $[d_j,b_j]$ and by this we get that $x_j(t)=0$ for every  $t\in (d_j,b_j]$. This argument implies that   $x_j(t)=0$ for every  $t\in (a_j,b_j]$. The conclusion follows by the arbitrariness of $j$.  
\end{proof}
\begin{rem}
    Lemma~\ref{lm:unique_extension} actually implies that assumption (H1)  holds. 
\end{rem}

%%%%%%%%%%%%%%%%%%%%%%%%%%%%%%%%%%%%%%%%%%%%%%%%%%%%%%%%%%%%%%%%%
%%
%%
%%
%%
%%
%%
%%
%%%%%%%%%%%%%%%%%%%%%%%%%%%%%%%%%%%%%%%%%%%%%%%%%%%%%%%%%%%%%%%%%

\subsection{A spectral flow formula on graphs}
Let $ H=L^2(\Gr)$. Then we have the decomposition:
	\begin{equation}\label{eq:splitting}
	\dom(L_s ^*) =\dom( L_s ) \oplus U 
	\end{equation}
where $U$ is a  $2(N_a+N_b)$-dimensional closed subspace of $H$.  We endow $U$ with the symplectic structure $\rho$ defined by restricting $\omega$ to $U$; in symbols
\[
\rho:= \omega|_U
\]
and we let 
\begin{equation}\label{eq:PIs-Vs}
p:\dom( L_s^*)\longrightarrow  U \qquad \textrm{ and }  \qquad V_s=\ker L_s^* 
\end{equation}
where $p$ denotes  the canonical projection onto the second factor. 

The following lemma collects the continuity properties needed in the sequel. 
They are all natural consequences of the stability theory of closed operators under perturbation, expressed in the gap topology, together with the corresponding continuity properties of graphs and induced subspaces; compare, for instance, Kato \cite[Chapter~IV, \S 2]{Kat80}. 
Recall that the gap topology on closed subspaces is the natural topology induced by the operator-norm distance between the corresponding orthogonal projections, and it provides the appropriate notion of continuous variation in the Grassmannian. 
In particular, once the family of minimal operators varies continuously in the gap sense, the associated adjoint graphs, projected boundary spaces, and the realizations determined by a gap-continuous family of Lagrangian subspaces inherit the same continuity behaviour.

\begin{lem}\label{thm:collect}
We assume condition  (H2).  Then the following hold.
	\begin{enumerate}
		\item The  path of minimal operators $s \mapsto L_s$ is gap-continuous in $H$.
		\item The path $s\mapsto \Graph\left(L_s^*|_U\right)$ is gap-continuous in $\Grass(H\times H)$.
		\item  $s\mapsto p(V_s)$ is continuous in $\Grass(U)$.
		\item  Let $s\longmapsto \Lambda_s \in \Lag(U, \rho)$ be a gap-continuous path. Then $L_{s,\Lambda_s}:= L_s^*|_{\dom( L_s ) \oplus \Lambda_s}$ is gap-continuous in $\Grass(H\times H)$.
	\end{enumerate}
\end{lem}
\begin{proof} 
By construction, for proving this result it's enough to restrict to a single edge, let's say the $j$-th. The conclusion follows from the general theory of one-sided singular SL-operators on $I_j$ constructed by authors at \cite{HPW25}. 
\end{proof}
Let  $\mathcal F:=\Set{f_1, \ldots, f_{2(N_a+N_b)}}$ be a basis of $U$ and we consider  the coordinate map  $O: U\to T^*\R^{N_a}\oplus T^*\R^{N_b}$ given by  $O^{-1}(e_i)=f_i$. We define on $T^*\R^{N_a}\oplus T^*\R^{N_b}$ the push-forward symplectic form given by  
\[
\rho^O(\cdot, \cdot):=\rho(O^{-1}\cdot ,O^{-1}\cdot)
\]
We observe that if  $s \mapsto p(V_s)$ is continuous, then the  mapping 
\[
 s\mapsto O( p(V_s))
\]
is a gap-continuous path of Lagrangian subspaces of  $(T^*\R^{N_a}\oplus T^*\R^{N_b},\rho^O)$.
We are now in position to state and to prove an abstract {\bf spectral flow formula on graphs}. Although this formula for graphs is new the proof is quite similar to the corresponding formula proved by authors in \cite{HPW25} and we include here for the sake of completeness. 
\begin{prop}\label{thm:Sturm_Sf_formula} 
Let $\Gr$ be a graph and we assume condition  (H2) holds  and that $s\mapsto \Lambda_s \in \Lag(U, \rho)$ is a gap continuous Lagrangian path. Then  the following spectral flow formula holds:
  \[
  \spfl ( L _{s,\Lambda_s},s\in I)= -\iCLM(O(\Lambda_s),  O( p(V_s)) ,\rho^O, s\in  I).
  \]
 \end{prop}

\begin{proof}
As direct consequence of Lemma~\ref{thm:collect},  we get that the map   $s\mapsto  L _{s,\Lambda_s}$ is  gap continuous  and by the previous discussion also the two Lagrangian maps $s\mapsto O(\Lambda_s)$ and  $s \mapsto O( p(V_s))$ are  continuous.  We also observe that the map $(s,r)\mapsto L_s+r\Id$ is continuous with gap topology being a bounded perturbation of the closed operator $L$. 
 
By using the localization property of the Maslov index and of the spectral flow, for  concluding the proof we only need to prove that the formula holds in a sufficiently small neighborhood of an instant $s\in I$. Without loss of generality, we can localize in a sufficiently small neighborhood of $0$.  Let's start by choosing   $\varepsilon >0$ such that $ L _{0,\Lambda_0}+\varepsilon \Id $ is invertible. Since the invertibility is an open condition, there exists $\delta>0$ such that $ L _{s,\Lambda_s}+\varepsilon \Id $ is invertible for  every $s\in [0,\delta]$. Then by the homotopy invariance property of spectral flow and of the Maslov index, we get that
\begin{multline}
	\spfl( L _{s,\Lambda_s},s\in [0,\delta])=\spfl( L _{0,\Lambda_0}+r\Id  , r\in [0,\varepsilon ])-\spfl( L _{\delta,\Lambda_\delta}+r\Id  , r\in [0,\varepsilon ])\quad \textrm{ and }\\[5pt]
	\iCLM(O(\Lambda_s) , O( p(V_s)),\rho^O, s\in [0,\delta])=\iCLM
	(O(\Lambda_{0}),O(p(V_{0,r})),\rho^O,  r\in [0,\varepsilon ])-\\
		\iCLM(O(\Lambda_{\delta}),O(p(V_{\delta,r})),\rho^O, r\in [0,\varepsilon ]).
\end{multline}
By this computation and by invoking Proposition~\ref{thm:main1-abstract}, about the spectral flow formula for operator's pencils, we immediately get that
\[
\spfl( L _{s,\Lambda_s},s\in [0,\delta])=-\iCLM(O(\Lambda_s),O(p(V_s)),\rho^O, s\in [0,\delta]).
\]
The conclusion  follows by summing up all (finite number) of local contributions to the Maslov index and to the spectral flow.
\end{proof}

%%%%%%%%%%%%%%%%%%%%%%%%%%%%%%%%%%%%%%%%%%%%%%%%%%%%%%%%%%%%%%%%%
%%
%%
%%
%%
%%
%%
%%
%%%%%%%%%%%%%%%%%%%%%%%%%%%%%%%%%%%%%%%%%%%%%%%%%%%%%%%%%%%%%%%%%

\subsubsection*{An explicit construction of the coordinate map}

This paragraph is devoted to provide an explicit construction of the abstract coordinate map $O$ defined above, in terms of the solution space of the Sturm-Liouville boundary value problem. This is done by constructing an ad-hoc {\sc trace map}.

Let $l$ denote the Sturm-Liouville operator given by
\begin{equation}\label{eq:Sturm-Liouville-operator}
		\ell:=-\dfrac{d}{dt}\left(P(t)\dfrac{d}{dt}+ Q(t)\right)+ \trasp{Q}(t)\dfrac{d}{dt}+ R(t) \qquad \textrm{ for } \quad t\in \Gr. 
\end{equation}
Denoting by $L$ and $L^*$ respectively the minimal and the  maximal operators induced by $l$, we define the {\em trace map} on $\dom(L^*)$ as follows: 
\begin{equation}\label{eq:trace-map}
\Tr: D(L^*) \to T^*\R^{N_a}\oplus T^* \R^{N_b} \quad \textrm{ defined by  } \Tr(f)=\big(\ev_a(f^{[1]}),\ev_a(f), \ev_b(f^{[1]}), \ev_b(f)\big).
\end{equation}
For each $i=1, \ldots, m$, we set 
\begin{itemize}
\item $Z_{i, a_i}$ be the vector subspace generated by  smooth functions $z_k$ defined on $I_i$ such that $(z_k^{[1]}(a_i), z_k(a_i))$ form a standard basis of $\R^{2d}$ and vanishes identically on a neighborhood of $b_i$, if $a_i<\infty$
\item $(0)$ if $a_i=-\infty$
\end{itemize}
For each $j=1, \ldots, m$, we set 
\begin{itemize}
\item $Z_{j, b_j}$ be the vector subspace generated by  smooth functions $z_k$ defined on $I_j$ such that $(z_k^{[1]}(b_j), z_k(b_j))$ form a standard basis of $\R^{2d}$ and vanishes identically on a neighborhood of $a_j$, if $b_j<\infty$
\item $(0)$ if $b_j=\infty$
\end{itemize}
For each $i=1, \ldots, m$, we let $Z_i\=Z_{i, a_i}\oplus Z_{i, b_i}$ and we set 
\[
U= \bigoplus_{i=1}^m Z_i.
\]
Let us now consider the symplectic space $(U, \rho)$ and we observe that with the above choice of the basis for $U$ the coordinate map $O$ reduces to the mapping $\Tr_U$ which is the restriction to $U$ of the map $\Tr$ defined at Equation~\ref{eq:trace-map}:
\[
\Tr_U: U \to T^*\R^{N_a}\oplus T^* \R^{N_b} \quad \textrm{ defined by  } \Tr_U(f)=\big(\ev_a(f^{[1]}),\ev_a(f), \ev_b(f^{[1]}), \ev_b(f)\big).
\]
Summing  up the previous arguments we get the following result. 
\begin{thm}\label{thm:Sturm_Sf_formula-SL} 
We assume condition  (H2) holds and let $s\mapsto\Lambda_s \in  \Lag(U,\rho)$ be a gap-continuous path.   Then  the following formula holds:
  \begin{equation}\label{eq:spfl-formula-trace}
\spfl ( L _{s,\Lambda_s},s\in I)= -\iCLM(\Tr_U(\Lambda_s), \Tr_U p(V_s), \rho^{\Tr_U} , s\in  I)
\end{equation}
where $\rho^{\Tr_U}$ denotes the push-forward of the symplectic structure $\rho$ through the trace map $\Tr_U$. 
 \end{thm}
 \begin{proof}
     The proof of this result is direct consequence of Proposition~\ref{thm:Sturm_Sf_formula} and the definition of the trace map. 
 \end{proof}
 \begin{rem}
     We observe that in the case of regular Sturm-Liouville boundary value problems, the Maslov intersection index appearing at the  (RHS) of Equation~\eqref{eq:spfl-formula-trace} coincides with the usual Maslov intersection index between the Lagrangian path induced by the family of monodromy matrices of the phase flow and the path of  Lagrangian subspaces parametrizing the boundary conditions. 
 \end{rem}

%%%%%%%%%%%%%%%%%%%%%%%%%%%%%%%%%%%%%%%%%%%%%%%%%%%%%%%%%%%%%%%%%
%%
%%
%%
%%
%%
%%
%%
%%%%%%%%%%%%%%%%%%%%%%%%%%%%%%%%%%%%%%%%%%%%%%%%%%%%%%%%%%%%%%%%%

\subsection{The Morse index theorem  on graphs}

Theorem~\ref{thm:Sturm_Sf_formula-SL} provides a general  spectral flow formula for a family of Sturm-Liouville differential operators defined on a garph $\Gr$ and it represents a crucial step for proving a {\bf Morse index theorem} in this  context. 

For  each $j=1, \ldots, m$ and for each $\sigma_j \in I_j$, the restriction of the differential expression $l_j$ on  $\mathscr C_0^\infty((a_j,\sigma_j), \R^d)$ is denoted by $l_{j,\sigma_j}$. Let us denote by $L_{j,\sigma_j}$ the minimal SL-operator associated to $l_{j,\sigma_j}$ and given the Dirichlet Lagrangian subspace $\Lambda_D \in \Lag(T^*\R^{N_a}\oplus T^*\R^{N_b},-\Omega_a\oplus \Omega_b)$, 
we  consider the decomposition
	\[
	\dom (L_{j,\sigma_j}^*):=\dom (L _{j,\sigma_j})\oplus Z_{j,\sigma_j}\qquad \textrm{ for }\qquad j=1, \ldots, m. 
	\]
where $Z_{j,\sigma_j}$ are constructed as before on each interval $a_j, \sigma_j)$ and we set
\[
L_{j,\sigma_j,\Lambda_D}\=L_{j,\sigma_j}^*|_{\Tr_{\sigma_j}^{-1}(\Lambda_D)}\qquad \textrm{ for }\qquad j=1, \ldots, m
\]
the selfadjoint extension, where we denoted by $\Tr_{\sigma_j}$ the map $\Tr_{\sigma_j}: \dom(L_{j,\sigma_j})^* \to T^*\R^{N_a} \oplus T^* \R^{N_b}$ defined at Equation~\eqref{eq:trace-map}.  
\begin{rem}
    It's worth observing that 
    \begin{multline}
        \dom(L_{j,\sigma_j}^*)=\bigoplus_{j=1}^m W^{2,2}((a_j, \sigma_j), \R^d) \qquad  \dom(L_{j,\sigma_j})=\bigoplus_{j=1}^m W_0^{2,2}((a_j, \sigma_j), \R^d) \qquad \textrm{ and } \qquad \\
        \dom(L_{j,\sigma_j,\Lambda_D})=\Set{(f_1, \ldots, f_m):=f \in \bigoplus_{j=1}^mW^{2,2}((a_j, \sigma_j))| f(a_j)=f(\sigma_j)=0}.
    \end{multline}
\end{rem}
% We let $\displaystyle Z_\sigma= \bigoplus_{j=1}^m Z_{j,\sigma}$ and we assume that 
% 	$\Tr(Z_{\sigma})=\Lambda_D \in \Lag(T^*\R^{N_a}\oplus T^*\R^{N_b},-\Omega_a\oplus \Omega_b)$.

\begin{thm}{\bf [Morse Index Theorem]}\label{thm:main-oneside}
Let $\Gr$ be the   graph  equipped by a Lagrangian function satisfying  conditions (L1)-(L2)-(H2). Let $x$ be a solution of the Euler-Lagrange equation under  Dirichlet boundary conditions at the  vertices. Then we have
\[
\iMor ( L _{\Gr,\Lambda_D})= \sum_{j=1}^m\sum_{\sigma \in (a_j,b_j)} \dim \ker  L_{j, \sigma,\Lambda_D}.
\]
\end{thm}
\begin{proof}
We start by observing that
\[
L_{\Gr,\Lambda_D}=\bigoplus_{j=1}^m L_{I_j,\Lambda_D}
\]
and $\dom(L_{\Gr,\Lambda_D})= \bigoplus_{j=1}^m \dom(L_{I_j,\Lambda_D})$. 
So, we only need to prove 
\[
\iMor(L_{I_j,\Lambda_D})=\sum_{\sigma \in (a_j,b_j)} \dim \ker  L_{j, \sigma,\Lambda_D}.
\]
The conclusion follows be the Morse index Theorem proved in  \cite{HPW25}. 
\end{proof}

%%%%%%%%%%%%%%%%%%%%%%%%%%%%%%%%%%%%%%%%%%%%%%%%%%%%%%%%%%%%%%%%%
%%
%%
%%
%%
%%
%%
%%
%%%%%%%%%%%%%%%%%%%%%%%%%%%%%%%%%%%%%%%%%%%%%%%%%%%%%%%%%%%%%%%%%

\subsection{Difference of Morse index}

The aim of this section is to provide a Morse Index theorem for general boundary conditions. Let $A\in \Cl^s(L^2(\Gr))$ and let 
\[ 
q_A[x]=\langle Ax ,x \rangle \qquad x \in \dom(A).
\]
and we assume  $q_A$ is lower bounded\footnote{Meaning that 
\[ 
q_A[x] \ge C\, \norm{x}^2\qquad x \in \dom(A)
\]
for some positive constant $C$.
}.
By following the standard Friedrich's extension construction arguments we get the quadratic form  $t_A$ which is the closure of $q_A$ with respect to the Friedrich's inner product.  

Let $\Lambda$ be  a Lagrangian  subspace of $(U,\rho)$ and we let $L_\Lambda$ be the self-adjoint extension of the Sturm-Liouville operator $L$ and we denote by $t_L$ and by $t_{L_\Lambda}$  the quadratic forms constructed as before by setting $A=L$ or $A=L_\Lambda$. 
% observe that for every $u,v\in \dom( L_\Lambda)$, the bilinear form  defined by  $\langle L_\Lambda u,v\rangle_{L^2}$  is lower bounded and symmetric. By taking the  closure of $\dom(L_\Lambda)$ in $W^{1,2}(\Gr,\R^d)$, we extend this symmetric bilinear form and  we denote it by $t_{L_\Lambda}$. So, in particular 
% $t_{ L _E}(\cdot,\cdot)$ is continuous on its domain denoted by  $\dom(t_{ L _E})$.
% Moreover, we have
% \[
% \dom (t_{ L _{D}})=W_0^{1,2}(\Gr,\R^d) \qquad \textrm{ and }  \dom (t_{ L _E})=\dom(t_{ L _{D}})+E.
% \]
% xwLet $ L^* :\dom(L ^*)\subset L^2(\Gr,\R^d)\to L^2(\Gr,\R^d)$  be a  SL-operator. 
\begin{lem}\label{lem:orth_complement_t_L}
	   Under the assumptions of Theorem~\ref{thm:main-oneside}, the following formula holds
 %    We assume that $\dom(L^*)=W^{2,2}(\Gr,\R^d)$, $L$ is a  Fredholm  operator and that the following decomposition holds: $\dom(L^*_j)=\dom(L_j)\bigoplus_{j=1}^m W_j$.
	% Let $W=\bigoplus_j W_j$. Then we have $\dom(L^*)=\dom(L)\oplus W$.
	
	%  Let $\Lambda$ be any Lagrangian subspace of $(W,\omega|_W)$ and let $\Lambda_D\subset W$ be the Dirichlet Lagrangian. Then the following equality holds
	\[
	[\dom(t_{ L })]^{t_{ L _\Lambda}} =\big[\dom(L_{\Lambda_D})+\dom( L _\Lambda)\big]\cap \ker  L ^*
	\]
    where $\dom(t_{ L })]^{t_{ L _\Lambda}}$ denotes the $t_{ L _\Lambda}$ orthogonal of $\dom(t_{L})$.
\end{lem}
\begin{proof}
	Since $ L $ has dense domain  in the Hilbert space $\dom(t_{ L })$ and $t_{ L _\Lambda}$ is a bounded form on $\dom(t_{ L _\Lambda})$, we have
	\[
	[\dom(t_ L )]^{t_{L_\Lambda}}=[\dom( L )]^{t_{ L _\Lambda}}=\{y\in \dom(t_{ L _{\Lambda}})|t_{ L _\Lambda}[x,y]=0 \textrm{ for all } x\in \dom( L )\}.
	\]
	Let $y\in \dom(t_{ L _\Lambda})$ and let $(y_n)_{n \in N}\subset \dom( L _\Lambda)$  be a sequence such that  $y_n\to y$ in $\dom(t_\mathcal {L_\Lambda})$. Then we have
	\[
	t_{ L _\Lambda}[x,y]=\lim_{n\to\infty} t_{ L _\Lambda}[x,y_n]=\lim_{n\to \infty} < L x,y_n>=< L x,y> \qquad \forall\, x\in \dom( L ).
	\]
	By this, we get that 
	It follows that 
	\begin{multline}
		[\dom(t_ L )]^{t_{ L _\Lambda}}=(\image  L )^\perp \cap \dom(t_{ L _\Lambda})=\ker  L ^*\cap \dom(t_{ L _\Lambda})\\=\ker  L ^*\cap(W^{1,2}_0(\Gr, \R^d)+\Lambda)=W^{2,2}(\Gr,\R^d)\cap (W^{1,2}_0(\Gr, \R^d)+\Lambda)\cap \ker L^*\\
        =(\Lambda + W^{2,2}(\Gr, \R^d) \cap W_0^{1,2}(\Gr, \R^d))\cap \ker L^*
		=(W_0^{2,2}(\Gr, \R^d)+\Lambda_D+\Lambda)\cap \ker L^*\\= 
		\big[\dom( L_{\Lambda_D})+\dom( L_\Lambda)\big]\cap \ker  L ^* .
	\end{multline}
	This concludes the proof.
\end{proof}

\begin{lem}\label{lem:index_orth_compl}
	Under the above notation,  the following holds
	\begin{equation}\label{eq:Q-restricted-WQ}
	\iMor\left(t_{ L_\Lambda}|_{[\dom(t_{L} )]^{t_{L_\Lambda}}}
	\right)=\coiMor\,\big[Q(p(\ker  L ^*),\Lambda;\Lambda_D)\big]
	\end{equation}
	where  $Q$ denotes the quadratic form appearing at Definition~\ref{def:kashi},
\end{lem}
\begin{proof}
	For $z_1,z_2 \in \ker  L ^*\cap [\dom( L _\Lambda)+\dom( L_{\Lambda_D})]$, let $z_i=p_i+q_i$  with $p_i\in \dom( L_\Lambda)$, $q_i\in \dom( L_{\Lambda_D})$.
	Then we have
	\[
	t_{ L_\Lambda}(z_1,z_2)=t_{ L_\Lambda}(p_1,p_2)+t_{ L_\Lambda}(p_1,q_2)+t_{ L_\Lambda}(q_1,p_2)+t_{ L_\Lambda}(q_1,q_2).
	\]
	Moreover
	\[
	t_{ L_\Lambda}(u,v)=\langle  L_\Lambda u, v\rangle_H \qquad \forall\, u\in \dom( L_\Lambda), \quad \forall\, v\in \dom(t_{ L_\Lambda}).
	\]	
	Since $t_{L_\Lambda}|_{\dom(t_{ L })}=t_{ L }$, then  we have
	\[
	t_{ L _\Lambda}(u,v)=t_{ L }(u,v)=\langle  L _D u,v \rangle_H \qquad  \forall\, u,v\in \dom( L_{\Lambda_D}).
	\]
	Now
	\begin{multline}
		t_{ L_\Lambda}(z_1,z_2)=\langle  L_\Lambda p_1,p_2\rangle+\langle  L _\Lambda p_1,q_2 \rangle +\langle  q_1,  L _\Lambda p_2 \rangle +\langle  L _D q_1,q_2 \rangle \\
		=\langle  L ^* p_1,p_2\rangle+\langle L^* p_1,q_2 \rangle +\langle  q_1,  L ^* p_2 \rangle +\langle  L ^* q_1,q_2 \rangle \\
		=( L ^*(p_1+q_1),p_2+q_2)-(p_1, L ^* q_2)+( L ^* p_1,q_2)
	\end{multline}
	Since $u_1=p_1+q_1\in \ker  L ^*$, then we have
	\[
	t_{ L _\Lambda}(z_1,z_2) =-(p_1, L ^* q_2)+( L ^* p_1,q_2)=-\omega(p_1,q_2).
	\]
	Let us now consider the bijection $p: \ker  L ^*\cap [\dom ( L_\Lambda) +\dom( L_{\Lambda_D})\big]\to p (\ker  L ^*)\cap(\Lambda +\Lambda_D)$.
	Then we can conclude that 	
	\[
	\iMor\left(t_{L_\Lambda}\Big\vert_{[\dom(t_L)]^{t_{L_\Lambda}}}
	\right)=\coiMor \Big[Q(p(\ker  L ^*),\Lambda;\Lambda_D)\Big]
	\]
	concluding the proof.
\end{proof}
\begin{lem}\label{thm:general_Morse _diff}
	Let $\mathcal Q$ be a quadratic form on linear space $\mathcal V$,  $\mathcal W$ be a closed  subspace of $\mathcal V$ and let 
	\[
	\mathcal W^{\mathcal Q}=\Set{v\in\mathcal V| b(w,v)=0,\  \forall w\in \mathcal W }
	\]
	where $b$ is the bilinear form induced by $\mathcal Q$ through the polarization identity and 	and we set 
	\[
	\ker \mathcal  Q= \mathcal V^{\mathcal Q}.
	\]
	We assume  that $\mathcal W^{\mathcal {QQ}}=\mathcal W+\ker \mathcal Q$ and that $\iMor(\mathcal Q|_{\mathcal W})<\infty$.  Then the following formula holds: 
	\[
	\iMor(\mathcal Q|_{\mathcal V})-\iMor(\mathcal Q|_{\mathcal W})=\iMor(\mathcal Q|_{\mathcal W^{\mathcal Q}}) +\dim\Big[ (\mathcal W\cap \mathcal W^{\mathcal Q}+\ker \mathcal Q)/\ker \mathcal Q\Big].
	\]
\end{lem}
\begin{proof}
For the proof we refer the interested reader to \cite[Theorem 3.1]{HWY20}. 
\end{proof}

\begin{thm}\label{thm:triple-index-Morse-Friedrich-extensions}
	We assume that $\iMor( L_{\Lambda_D})<\infty$.  Under the above notation, the following equality holds
	\[
	\iMor( L _\Lambda)-\iMor( L_{\Lambda_D})= \itriple(p(\ker  L ^*), \Lambda,\Lambda_D).
	\]
	where $\itriple$ denotes the triple index. (Cfr. Appendix~\ref{sec:Maslov} and references therein).
\end{thm}

\begin{proof}
	Let $\gamma:\dom(  L ^*)\to \beta( L )$ be the projection onto the factor space. By invoking  Lemma~\ref{thm:general_Morse _diff}, for proving the theorem, we only  need to compute
	\[
	\dim \Big[\big(\dom(t_{ L })\cap [\dom(t_{ L })]^{t_{ L _\Lambda}}+\ker t_{ L _\Lambda} \big)/\ker t_{ L _\Lambda}\Big].
	\]
	By arguing as in the proof of Lemma~\ref{lem:orth_complement_t_L}, we get that $\ker t_{ L _\Lambda}=\ker  L _\Lambda$ and then by Lemma~\ref{lem:orth_complement_t_L}, we get
	\begin{align}
		\dom(t_{ L })\cap [\dom(t_{ L })]^{t_{ L _\Lambda}}=\dom(t_ L )\cap \ker  L ^*\cap \dom( t_{ L _\Lambda})=\dom(t_ L )\cap \ker  L ^*=\ker  L_{\Lambda_D}.
	\end{align}
	Then, finally 
	\begin{multline}\label{eq:int-1}
		\big(\dom(t_{ L })\cap [\dom(t_{ L })]^{t_{  _\Lambda}}+\ker t_{ L _\Lambda} \big)/\ker t_{ L _\Lambda}\\=(\ker  L_{\Lambda_D}+\ker t_{ L _\Lambda})/\ker t_{ L_\Lambda}\cong \ker  L_{\Lambda_D}/(\ker  L_{\Lambda_D}\cap \ker  L_\Lambda).
	\end{multline}
	By counting dimensions, we infer that
	\begin{equation}\label{eq:int-3}
		\dim\Big[ \big(\dom(t_{ L })\cap [\dom(t_{ L })]^{t_{ L _\Lambda}}+\ker t_{ L _\Lambda} \big)/\ker t_{ L _\Lambda}\Big] =\dim \ker L_{\Lambda_D} -\dim(\ker  L_{\Lambda_D}\cap \ker  L_\Lambda).
	\end{equation}
	Let $p$ be the map $\dom( L ^*)\to U$, then we get
	\begin{equation}\label{eq:penultima}
	\dim\Big[\big(\dom(t_{ L })\cap [\dom(t_{ L })]^{t_{ L _\Lambda}}+\ker t_{ L _\Lambda} )/\ker t_{ L _\Lambda}\big] =\dim\big[p (\ker  L^*)\cap D\big] -\dim\big[  p(\ker  L^*)\cap \Lambda_D\cap \Lambda\big]
	\end{equation}
Now, by setting $\mathcal Q=t_{L_\Lambda}$, $\mathcal W=\dom(t_{L})$, we get that the (LHS) at Equation~\eqref{eq:Q-restricted-WQ} coincides with the term $\mathcal Q|_{\mathcal W^{\mathcal Q}}$. Moreover, we have 
\[
 (\mathcal W\cap \mathcal W^{\mathcal Q}+\ker \mathcal Q)/\ker \mathcal Q=([\dom(t_{ L })]^{t_{ L _\Lambda}}+\ker t_{ L _\Lambda} \big)/\ker t_{ L _\Lambda}.
\]
So, in particular
\begin{multline}\label{eq:ultima}
    \iMor(\mathcal Q|_{\mathcal W^{\mathcal Q}}) +\dim\Big[ (\mathcal W\cap\mathcal  W^{\mathcal Q}+\ker \mathcal Q)/\ker \mathcal Q\Big]\\ = \iMor\left(t_{L _\Lambda}|_{[\dom(t_{L} )]^{t_{L_\Lambda}}}\right)+ \dim \Big[\big(\dom(t_{ L })\cap [\dom(t_{ L })]^{t_{ L _\Lambda}}+\ker t_{ L _\Lambda} \big)/\ker t_{ L _\Lambda}\Big]\\
    =\coiMor\,\big[Q(p(\ker  L ^*),\Lambda;\Lambda_D)\big]  +
    \dim\big[p (\ker  L^*)\cap \Lambda_D\big] -\dim\big[  p(\ker  L^*)\cap \Lambda_D\cap \Lambda\big]
\end{multline}
    where the last equality follows by Equation~\eqref{eq:penultima} and Equation~\eqref{eq:Q-restricted-WQ}. Now, by taking into account Definition~\ref{def:kashi}, we get   that the last member of Equation~\eqref{eq:ultima} coincides with the triple index $\itriple(p(\ker  L ^*), \Lambda,\Lambda_D)$. By Lemma~\ref{thm:general_Morse _diff}, the first member of Equation~\eqref{eq:ultima} coincides with the following difference $  \iMor(\mathcal Q|_{\mathcal V})-\iMor(\mathcal Q|_{\mathcal W})$ where $\mathcal V=\dom(t_{L_\Lambda})$.  Let us now observe that $\mathcal V=\dom(T_{\Lambda_D})$. This is because $\dom (L)= W_0^{2,2}(\Gr, \R^d)$.  Since $t_L$ is equivalent to the norm of $W^{1,2}(\Gr, \R^d)$, we get that $\dom(t_L)=W_0^{1,2}(\Gr, \R^d)$.
The conclusion follows by observing that 
\[
\iMor(t_{L_\Lambda})=\iMor(L_\Lambda) \qquad \textrm{ and } \qquad \iMor(t_{L_{\Lambda_D}})=\iMor(L_{\Lambda_D}).
\]
\end{proof}

%{\color{blue}
\begin{lem}\label{lem:triple_diff_circle_permu}
\[\itriple(U,V,W)-\itriple(V,W,U)=\dim(U\cap W) -\dim(V\cap U)
\]
or equivalently 
\[
\itriple(U,V,W)-\dim(U\cap W)=\itriple(V,W,U) -\dim(V\cap U)
\]
\end{lem}
\begin{proof}
By the definition of triple index
\begin{align*}
\itriple(U,V,W)=\coiMor \big[  Q(U,V;W)\big]+\dim(U\cap W)-\dim(U\cap V\cap W)\\
\itriple(V,W,U)=\coiMor \big[  Q(V,W;U)\big]+\dim(U\cap V)-\dim(U\cap V\cap W).
\end{align*}
Since $\coiMor[Q]$ is invariant under circular permutation, the lemma follows.
\end{proof}

\begin{cor}\label{thm:differenza-bc}
Consider a decomposition $\dom(L^*)=\dom(L)\oplus U$ with $\ker L^*\subset U$. Let $\Lambda_0,\Lambda_1\in \Lag(U,\rho)$.
Let $\Lambda_D=U\cap \dom(L_{\Lambda_D})$ where $L_{\Lambda_D}$ is the Friedrich extension of $L$.
Then we have
	\[
	\iMor( L _{\Lambda_1})-\iMor( L_{\Lambda_0})= 
    \itriple(p(\ker L^*),\Lambda_1,\Lambda_0)-\itriple(\Lambda_1,\Lambda_0,\Lambda_D).
	\]
\end{cor}
\begin{proof}
By Theorem~\ref{thm:triple-index-Morse-Friedrich-extensions}, we get that 
\[
\iMor(L_{\Lambda_i})-\iMor(L_{\Lambda_D})= \itriple(p(\ker L^*), \Lambda_i,\Lambda_D),\ i=0,1.
\] 
By this, using Lemma \ref{lem:triple_diff_circle_permu} and Proposition \ref{thm:mainli}, we get that
\begin{align}
\iMor(L_{\Lambda_1})-\iMor(L_{\Lambda_0})&=\itriple(\ker L^*,\Lambda_1,\Lambda_D)-\itriple(p(\ker L^*,\Lambda_0,\Lambda_D)\\
&=[\itriple(\Lambda_D,\ker L^*,\Lambda_1)-\dim(\Lambda_D\cap\Lambda_1)]-[\itriple(\Lambda_D,\ker L^*,\Lambda_0)-
\dim(\Lambda_D\cap \Lambda_0)]\\
&=s(\Lambda_D,\ker L^*,\Lambda_0,\Lambda_1)-\dim(\Lambda_D\cap\Lambda_1)+\dim(\Lambda_D\cap\Lambda_0)
\end{align}
Using proposition \ref{thm:mainli} again, we get
\begin{align*}
s&(\Lambda_D,\ker L^*,\Lambda_0,\Lambda_1)-\dim(\Lambda_D\cap\Lambda_1)+\dim(\Lambda_D\cap\Lambda_0)\\
&=-s(\Lambda_D,\ker L^*,\Lambda_1,\Lambda_0)-\dim(\Lambda_D\cap\Lambda_1)+\dim(\Lambda_D\cap\Lambda_0)\\
&=-\itriple(\Lambda_D,\Lambda_1,\Lambda_0)+
\itriple(\ker L^*,\Lambda_1,\Lambda_0)-\dim(\Lambda_D\cap\Lambda_1)+\dim(\Lambda_D\cap\Lambda_0)\\
&=\itriple(\ker L^*,\Lambda_1,\Lambda_0)-\itriple(\Lambda_1,\Lambda_0,\Lambda_D).
\end{align*}
The last equation comes from Lemma \ref{lem:triple_diff_circle_permu}.
\end{proof}

\begin{rem}
If $\Lambda_1$, $\Lambda_0$ are conormal boundary condition associated with $V_1$, $V_0$, then
\begin{align}
\itriple(\Lambda_1,\Lambda_0,\Lambda_D)&=\dim (\Lambda_1\cap\Lambda_D)-\dim(\Lambda_1\cap \Lambda_0)\cap \Lambda_D\\&=\dim V_1^\perp -\dim(V_1^\perp\cap V_0^\perp)=2n-\dim V_1-(2n-\dim(V_1+V_0))\\
&=\dim (V_1+V_0)-\dim V_1.
\end{align}
\end{rem}
\begin{rem}
It is worth noting that Corollary~\ref{thm:differenza-bc} extends to non-compact graphs the main result established by the authors in \cite[Theorem~1, p.~2798]{ABB23}. See also Remark~\ref{rem:Agrachev-comparison} for the connection between the triple index and the {\em negative Maslov index} introduced by the authors in the aforementioned paper.
\end{rem}

%%%%%%%%%%%%%%%%%%%%%%%%%%%%%%%%%%%%%%%%%%%%%%%%%%%%%%%%%%%%%%%%%
%%
%%
%%
%%
%%
%%
%%
%%%%%%%%%%%%%%%%%%%%%%%%%%%%%%%%%%%%%%%%%%%%%%%%%%%%%%%%%%%%%%%%%

\section{Some applications to star graphs}\label{sec:star-graphs}

This section is dedicated to prove 
\begin{enumerate}
    \item A reduction formula for the union of two start graphs  both having one 1 leaf vertex (i.e. a vertex with just one adjacent edge) and we assume that only one of two is non-compact 
    \item Some formulas measuring  the difference between the Morse and the Maslov index for a star graph or a union of two star graphs, in terms of the number of leaf vertices. 
\end{enumerate}
The first formula provides a proof of the crucial importance  of the Kirchhoff condition in the index theory on graphs and give some insight in studying special orbits (e.g. brake orbits) in Hamiltonian  systems. 

The second formula allows us to  characterize from a variational viewpoint a quantum {\sc star graph}, let's say to establish the number of edges by knowing the Morse and the Maslov index defined on it. Moreover if the properties of the graph are known, it provides a way to compute the Morse index through the Maslov index in a quite effective way.

%%%%%%%%%%%%%%%%%%%%%%%%%%%%%%%%%%%%%%%%%%%%%%%%%%%%%%%%%%%%%%%%%
%%
%%
%%
%%
%%
%%
%%
%%%%%%%%%%%%%%%%%%%%%%%%%%%%%%%%%%%%%%%%%%%%%%%%%%%%%%%%%%%%%%%%%

\subsection{A reduction formula for the union of two star graphs}

We consider a graph $\Gr$ with a single central vertex to which one bounded edge and one unbounded edge are attached. We write $\GStar{1}$ for the bounded edge and $\GStar{1}^\infty$ for the unbounded one. On the graph $\Gr$ the following one-parameter family of Sturm-Liouville operators has been given: 
\begin{equation}\label{eq:Sturm-Liouville-operator-Gr-easy}
		l_s:=-\dfrac{d}{dt}\left(P(t)\dfrac{d}{dt}+ Q(t)\right)+ \trasp{Q}(t)\dfrac{d}{dt}+ R(t) + C_s(t)\qquad \textrm{ for } \quad (t,s)\in \Gr\times I 
\end{equation}
where  the restriction onto each arc  $I_j$ of the graph, i.e.  $l_s|_{I_j}$  is given by 
\begin{equation}\label{eq:Sturm-Liouville-operator-graph-easy}
		l_{j,s}:=-\dfrac{d}{dt}\left(P_j(t)\dfrac{d}{dt}+ Q_j(t)\right)+ \trasp{Q}_j(t)\dfrac{d}{dt}+ R_j(t) + C_{j,s}(t)\qquad \textrm{ for } \quad (t,s)\in I_j \times I \qquad j=1,2
\end{equation}
and where $I_1=[0,1]$ and $I_2=[1,+\infty)$. We assume that condition (H2) throughout the whole section. 

The minimal domain of the operator $L_s$ associated to $l_s$ has fixed ($s$-independent) domain given by 
\[ 
\dom(L)=W_0^{2,2}(I_1, \R^d) \oplus W_0^{2,2}(I_2, \R^d).
\]
We choose the  orientation on $\Gr$ from $0$ to $1$ for the first edge and from $1$ to $+\infty$ for the second edge. Even if all the results are independent on the chosen orientation, we need to chose it once and for all in order to  properly define the sign of the symplectic form. 

Let us consider the symplectic space defined by $(\R^{2d} \oplus\R^{2d} \oplus\R^{2d}, (-\Omega\oplus \Omega) \oplus-\Omega)$. In the previous notation, the parenthesis emphasize the role of the  boundary condition on the symplectic space.  We denote by $L_s^*$ the maximal operator associated to $l_s$ and we consider the projection map $\gamma: \dom(L_s^*) \to \beta(L_s)$ onto the Gelfand-Robbin quotient. We observe that, also in this case the operator $L^*_s$ associated to $l_s$ has fixed ($s$-independent) domain given by
\[ 
\dom(L^*)=W^{2,2}(I_1, \R^d) \oplus W^{2,2}(I_2, \R^d).
\]
By the definition of $\gamma$, we get that 
\[
\gamma(\ker L^*_s)= \Graph(M_s(1))\oplus W_s^{st}(1)
\]
where  \(M_{s}(1)\) denotes the fundamental solution matrix of the Hamiltonian  system associated to the operator $l_{1,s}$ on $[0,1]$ whilst $W_s^{st}(1)$ denotes the stable space of the Hamiltonian  system induced by the operator $l_{2,s}$ on $[1,+\infty)$.

Let $\Lambda$ any Lagrangian subspace of the above defined symplectic space. By invoking Theorem~\ref{thm:Sturm_Sf_formula-SL}, we get that the following spectral flow formula holds: 
\begin{equation}\label{eq:spf-formula-grafo-fava-1}
\spfl\left(L_{s, \Lambda}, s \in [0,1]\right)=-\iCLM\left(\Lambda, \Graph(M_{s}(1)) \oplus W_{s}^{st}(1), s\in [0,1]\right).
\end{equation}
A particularly interesting  case is the one obtained for $\Lambda=L_D \oplus \Delta\in \Lag(\R^{2d}, -\Omega\oplus \Omega)\oplus \Lag(\R^{2d}, -\Omega)$, where $\Delta=\Graph(\Id_d)$  denotes the  Lagrangian corresponding to the continuity condition of the linear dynamical system at the vertex $t=1$. 

We now prove an interesting result which states that posing Kirchhoff condition at the instant $t=1$ and Dirichlet boundary condition at the leaf vertex of the bounded graph $\GStar{1}$, reduced the problem on $\Gr$ to a problem on the half-line $[0,+\infty)$. 
The following reduction formula holds. 
\begin{prop}
    Under the above notation and assumption  (H2), we get that 
    \[
    \spfl\left(L_{s, \Lambda}, s \in [0,1]\right)=-\iCLM\left(\Lambda_D, W_s^{st}(0), s\in [0,1]\right).
    \]
\end{prop}
\begin{proof}
    By Equation~\eqref{eq:spf-formula-grafo-fava-1}, we get that 
    \begin{equation}\label{eq:spf-formula-grafo-fava-2}
\spfl\left(L_{s, \Lambda}, s \in [0,1]\right)=-\iCLM\left(\Lambda_D\oplus\Delta, \Graph(M_{s}(1)) \oplus W_{s}^{st}(1), s\in [0,1]\right).
\end{equation}
By using the symplectic reduction, we get that (RHS) of Equation~\eqref{eq:spf-formula-grafo-fava-2}, is equal to 
\[
-\iCLM\left(\Lambda_{D}, R_{0 \oplus \Delta}\left(\Graph\left(M_{s}(1)\right) \oplus W_{s}^{st}(1)\right), s \in [0,1]\right). 
\]
We observe that 
\begin{multline}
    R_{0 \oplus \Delta}\left(\Graph\left(M_{s}(1)\right) \oplus W_{s}^{st}(1)\right)=\left(\left(\Graph\left(M_{s}(1)\right) \oplus W_{s}^{st}(1)\right) \cap\left(\mathbb{R}^{2 n} \oplus \Delta\right)+(0 \oplus \Delta)\right) /(0 \oplus \Delta)\\=W_{s}^{st}(0) \oplus \Delta/(0 \oplus \Delta)
\end{multline}
where the last equality can be proved as follows. 
Let \(T_{s}=\Graph\left(M_{s}(1)\right)\) and let 
\[
\left(u, T_{s} u, w\right) \in \Graph\left(M_{s}(1)\right) \oplus W_{s}^{st}(1). 
\]
Here \(u \in \mathbb{R}^{2 d}\) and \(w \in W_{s}^{st}(1)\). Now, if \(\left(u, T_{s} u, w\right) \in \mathbb{R}^{2 d} \oplus \Delta\), then we have \(T_{s} u=w\). So, 
\[\left(\left(\Graph\left(M_{s}(1)\right) \oplus W_{s}^{st}(1)\right) \cap\left(\mathbb{R}^{2 n} \oplus \Delta\right)=\left\{\left(T_{s}^{-1} w, w, w\right) \mid w \in W_{s}^{s}(1)\right\}\right.
\]
By this, we finally get that 
\begin{equation}\label{eq:fava-finale-1}
\iCLM\left(\Lambda_{D}, R_{0 \oplus \Delta}\left(\Graph\left(M_{s}(1)\right) \oplus W_{s}^{u}(1)\right), s \in [0,1]\right)= 
\iCLM\left(\Lambda_{D}, W_{s}^{st}(0), s \in [0,1]\right).
\end{equation}
The conclusion follows by Equation~\eqref{eq:fava-finale-1} and  by Equation~\eqref{eq:spf-formula-grafo-fava-2}. This concludes the proof. 
\end{proof}

%%%%%%%%%%%%%%%%%%%%%%%%%%%%%%%%%%%%%%%%%%%%%%%%%%%%%%%%%%%%%%%%%
%%
%%
%%
%%
%%
%%
%%
%%%%%%%%%%%%%%%%%%%%%%%%%%%%%%%%%%%%%%%%%%%%%%%%%%%%%%%%%%%%%%%%%

% \subsection{An index formula for compact graphs}

% This section is devoted to explicitly compute the difference between the Morse and the Maslov index in the case of a Legendre convex problem on a compact graph having finitely many edges and vertices. The building blocks for constructing all of these graphs are the star graphs. So, we start first to analyze the star graphs and then we consider a more general situation. 

%%%%%%%%%%%%%%%%%%%%%%%%%%%%%%%%%%%%%%%%%%%%%%%%%%%%%%%%%%%%%%%%%
%%
%%
%%
%%
%%
%%
%%
%%%%%%%%%%%%%%%%%%%%%%%%%%%%%%%%%%%%%%%%%%%%%%%%%%%%%%%%%%%%%%%%%

\subsection{The Morse index theorem for a compact star graph}

A {\sc star graph} is one of the simplest graphs in mathematics. It consists of a single central vertex connected to several leaf vertices, making it a prime example of a tree, a complete bipartite graph, and a bipartite network with extreme centrality. Because of its  straightforward structure, it  is the first step in order to classify the graphs through their variational properties.  

Let us consider the standard symplectic space 
\(\left(\mathbb{R}^{n}\times \R^n, \Omega\right)\) equipped with the system of canonical coordinates $(p,q)$ where $p$ represents the {\sc conjugate  momentum} and $q$ the {\sc configuration variable}. We doubling the symplectic space by flipping the sign of the symplectic form in the first component: 
\(\left(\mathbb{R}^{2 n} \times \mathbb{R}^{2 n},-\Omega \times \Omega\right)\) and let us consider the $k$-dimensional subspace $V$ (here $k=0, \ldots, 2n$)   of the configuration space. We denote by $V^\perp$ its orthogonal. In the doubled symplectic space, let us consider the {\sc Neumann Lagrangian subspace} $L_N$ defined by 
\[
L_{N}=(0) \times \mathbb{R}_{q}^{n}\times(0) \times \mathbb{R}_{q}^{n}.
\]
So, the subspaces $V, V^\perp$ can be identified respectively with the  subspaces  \(I_V, I_{V^\perp}\)  of $L_N$ defined respectively by: 
\[
I_{V}=\Set{\left(0, q_{1}, 0, q_{2}\right) | \left(q_{1}, q_{2}\right) \in V} \quad \textrm{ and } \quad 
I_{V^\perp}=\Set{\left(0, q_{1}, 0, q_{2}\right)|\left(q_{1}, q_{2}\right) \in V^{\perp}}
\]
We let \(\mathcal J=\begin{bmatrix}-J & 0 \\ 0 & J\end{bmatrix}\). The the following result holds. 
\begin{lem}
The subspace $\Lambda_V$ defined by 
\[
\Lambda_{V}=\mathcal J\left(I_{V^{\perp}}\right) \oplus I_{V} .
\]
is Lagrangian.
\end{lem}
\begin{proof} We observe that \(\mathcal J\left(I_{V^{\perp}}\right)<L_{D}=\mathbb{R}^{n} \times(0) \times \mathbb{R}^{n} \times(0)\) where $L_D$ denotes the {\sc Dirichlet Lagrangian}. Now,  since \(L_{D}\) is a Lagrangian subspace, it follows that  \(\mathcal J\left(I_{V^{+}}\right)\) is isotropic. Let
\begin{multline}
u_{1}+v_{1} \in \Lambda_{V} \qquad \textrm{ and } \qquad  u_{2}+v_{2} \in \Lambda_{V}\\ \textrm{ for  } \qquad 
u_{1}, u_{2} \in \mathcal J(I_{V^\perp})\qquad  \textrm{ and } \qquad   v_{1}, v_{2} \in I_{V}.
\end{multline}
Then, we have:
\begin{multline}
 (-\Omega\times\Omega)\left(u_{1}+v_{1}, u_{2}+v_{2}\right)=(-\Omega\times \Omega)\left(u_{1}+v_{1}, u_{2}+v_{2}\right)= \\[3pt]
 =\Omega\left(u_{1}, u_{2}\right)+\Omega\left(u_{1}, v_{2}\right)+\Omega\left(v_{1}, u_{2}\right)+\Omega\left(v_{1}, v_{2}\right)\\[3pt] 
 = \Omega\left(u_{1}, v_{2}\right)+\Omega\left(v_{1}, u_{2}\right)=0
\end{multline}
where the conclusion follows by observing that 
$J u_{1}, J u_{2} \in I_{V^\perp}$, $v_1, v_2 \in I_V$ and finally by observing that $I_V \perp I_{V^\perp}$.
\end{proof}
\begin{rem}
Another way to prove the lemma is the following. Let 
$\trasp{(u_1, v_1)}$ and $\trasp{(u_2, v_2)}$ two vectors in $\Lambda_V$. By a straighforward calculation, we get:
\begin{multline}
 (-\omega \times \omega)\left[\binom{u_{1}}{v_{1}},\binom{u_{2}}{v_{2}}\right]=\left\langle\left(\begin{array}{cc}
- J & 0 \\
0 & J
\end{array}\right)\binom{u_{1}}{v_{1}},\binom{u_{1}}{v_{2}}\right\rangle= \\
 =\left\langle\left(\begin{array}{c}
-J\, u_{1} \\
J v_{1}
\end{array}\right),\binom{u_{2}}{v_{2}}\right\rangle=-\left\langle J u_{1}, u_{2}\right\rangle+\left\langle J v_{1}, v_{2}\right\rangle=0
\end{multline}
since both $I_V$ and $\mathcal J(I_{V^\perp})$ are isotropic. 
\end{rem}

Let $\GStar{m+1}$ be the quantum directed graph having one central vertex and  $m$ leaf vertices, or $\mathcal K_{1,m}$  emphasizing its structure as a complete bipartite graph where one subset has one vertex and the other subset has $m$ vertices, and we assume that  $\GStar{m+1}$ is compact. We assume that the orientation has been chosen in such a way that the central point is the starting point of each arc and we consider the variational problem having conormal boundary conditions.  

\def\m{3} % set this to the number of leaves you want
\begin{figure}[ht]
  \centering
  \begin{tikzpicture}[>=stealth, every node/.style={draw,circle,inner sep=1.0pt}, node distance=1.0cm]
    % central vertex
    \node (c) at (0,0) {$v_0$};

    % place and draw each leaf and arrow
    \foreach \i in {1,...,\m} {
      \node (l\i) at ({360*(\i-1)/\m}:3cm) {$v_{\i}$};
      \draw[->] (c) -- (l\i);
    }
  \end{tikzpicture}
  \caption{Directed star graph with \(m=3\) leaves.}
  \label{fig:directed-star}
\end{figure}
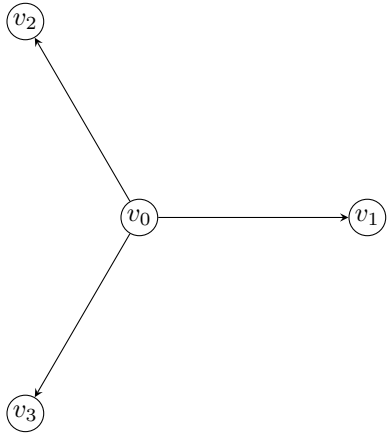
We start to consider the linear subspaces
\[
W_{j}<\mathbb{R}^{d} \qquad \textrm{ and }\qquad  V<\underbrace{\mathbb{R}^{d} \times \ldots \times \mathbb{R}^{d}}_{m \textrm {-times }}=\left(\mathbb{R}^{d}\right)^{m}.
\]
Here $W_j$ is a $k_j$-dimensional subspace for $k=0, \ldots, d$ and  $V$ is an $h$-dimensional for $h=0, \ldots md$. For eavery $j=1, \ldots, m$, we let 
\begin{equation}
N^{*}\left(W_{j}\right)= W_{j}^{\perp} \times W=\Set{(p, q) \in \mathbb{R}^{d} \times \mathbb{R}^{d}| q \in W_{j} \textrm{ and }   \langle p,q\rangle =0 \quad \forall q \in W_{j}} \in \Lag(\R^{2d}, \omega_j)
\end{equation}
and we defined 
\[
\Lambda_W= \bigoplus_{j=1}^m N^*(W_j) \in  \bigoplus_{j=1}^m \Lag(\R^{2d}, \omega_j).
\]
Denoting by \(\left(\mathbb{R}^{m d} \times \mathbb{R}^{m d}, \Omega\right)\) the standard $2md$-dimensional symplectic space, let us now define the following new symplectic space 
\[
\big(\R^{2md}\times \R^{2md}, \widetilde \Omega\big) \qquad \textrm{ where } \qquad \widetilde{\Omega}\=\bigoplus_{j=1}^{m} \omega_{j} \oplus -\Omega. 
\]
We let
\[
\widetilde{J}=\bigoplus_{j=1}^{m} J_{j} \oplus-J
\]
where $J_{i} \in\Sp(2 d)$, $J \in \Sp(2md)$ and finally $\widetilde J\in \Sp(4md)$. Let us now define the  Lagrangian subspace:
\[
\Lambda_V=\widetilde J(I_{V^\perp}) \oplus I_V \in \Lag(2md),
\]
where $I_{V^\perp}$ and $I_V$ are the isotropic subspace constructed as before. 

Now let \(\Lambda_{0}=\Lambda_W+\Lambda_V\) and  we let \(\Lambda_{D} \in \Lag(2md)\) be the Dirichlet boundary condition.  For $j=1, \ldots, m$, we denote  by \(\gamma_{j}\)    the fundamental matrix solutions of the  Hamiltonian  system along the $j$-th arc corresponding to the Sturm-Liouville operator $l_{j}$ given at Equation~\eqref{eq:Sturm-Liouville-operator} and we set $\Lambda_\gamma=\bigoplus_{j=1}^m \Graph(\gamma_j)$. Then we have 
\begin{multline}\label{eq:Morse-Maslov -star-graphs}
\iMor\left(\Lambda_{0}\right)-\iCLM\left(\Lambda_{0}\right)=\left[\iMor\left(\Lambda_{0}\right)-\iMor\left(\Lambda_{D}\right)\right]
+\left[\iMor\left(\Lambda_{D}
\right)-\iCLM\left(\Lambda_{D}\right)\right]\\[3pt]+\left[\iCLM\left(\Lambda_{D}\right)-\iCLM\left(\Lambda_{0}\right)\right]
=\itriple\left(\Lambda_\gamma, \Lambda_{0}, \Lambda_{D}\right)-m\, d+s\left(\Graph(\Id), \Lambda_\gamma; \Lambda_{0}, \Lambda_{D}\right) \\[3pt]
=-md+\itriple\left(\Graph(\Id), \Lambda_{0}, \Lambda_{D}\right)
\end{multline}
where we denoted by $\itriple$ and $s$ respectively the triple and the H\"ormander index. (Cfr. Appendix~\ref{sec:spectral-flow-Maslov-index} and references therein).

Formula provided at Equation~\eqref{eq:Morse-Maslov -star-graphs} it is very useful in the applications and it allows us  to characterize the number of arcs  of the {\sc star graphs} just by looking at the Morse and the Maslov indices.

%%%%%%%%%%%%%%%%%%%%%%%%%%%%%%%%%%%%%%%%%%%%%%%%%%%%%%%%%%%%%%%%%
%%
%%
%%
%%
%%
%%
%%
%%%%%%%%%%%%%%%%%%%%%%%%%%%%%%%%%%%%%%%%%%%%%%%%%%%%%%%%%%%%%%%%%

\subsubsection*{Dirichlet on the leaf vertices and Kirchhoff on the central vertex}

Now we assume that on the leaf vertices we have Dirichlet boundary conditions. So, by using the previous notation, we get  $W_j=(0)$ for every $j=1, \ldots, m$ and so the corresponding Lagrangian subspace $\Lambda_W$ is given  by 
\[
\Lambda_W=\bigoplus_{j=1}^m L_D^j \qquad \textrm{ where } \qquad L_D^j=\R^d \times (0) \in \Lag(\R^{2d}, \omega_j). 
\]
At the central vertex we assume  {\sc Kirchhoff boundary condition}. More precisely, we let 
\[
V=\Set{(q_1, \ldots, q_m)\in (\R^d)^m|q_1=\ldots=q_m}.
\]
and let $V^\perp$ its orthogonal defined by 
\[
V^\perp=\Set{(q_1, \ldots, q_m) \in (\R^d)^m|q_1+\ldots +q_m=0}.
\]
To $V$ and $V^\perp$ we associate as before the subspaces $I_V$ and $I_V^\perp$  and we consider the Lagrangian subspace $\Lambda_V$ defined as above, by
\[
\Lambda_V=\widetilde J(I_{V^\perp}) \oplus I_V \in \Lag(2md).
\]
The next step is to explicit compute the triple index appearing at Equation~\eqref{eq:Morse-Maslov -star-graphs}. To do so, we let 
\[
\widetilde V= V \oplus \underbrace{0\oplus \ldots \oplus 0}_{m \textrm {-times }} 
\]
and we observe that 
\[
\widetilde V^\perp= V^\perp \oplus \underbrace{\R^d\oplus \ldots \oplus \R^d}_{m \textrm {-times }} 
\]
By taking into account the formula given at \cite[Example 2.8]{HWY20}, we get that 
\begin{equation}\label{eq:triple-and-V}
\itriple\left(\Graph(\Id), \Lambda_{0}, \Lambda_{D}\right)= md-\dim\left[\widetilde  V^{\perp} \bigcap \bigoplus_{j=1}^m\Graph\left(-\Id_d\right)\right]\cong \widetilde  V^{\perp}.
\end{equation}
Since $\dim \widetilde  V^{\perp}= m(d-1)$, then we get that 
\begin{equation}\label{eq:contribution}
\itriple\left(\Graph(\Id), \Lambda_{0}, \Lambda_{D}\right)= m\,d- m\, (d-1)=m.
\end{equation}
We finally get that 
\[
\iMor\left(\Lambda_{0}\right)-\iCLM\left(\Lambda_{0}\right)=
-md+\itriple\left(\Graph(\Id), \Lambda_{0}, \Lambda_{D}\right)=-(m-1)\,d.
\]
From a variational viewpoint, we have the following data. 
\begin{itemize}
\item The directed compact graph  $\GStar{m+1}$ having one central starting vertex and  $m$ endpoints leaf vertices
\item A regular Lagrangian function $\mathscr L: \GStar{m+1} \times T\R^d \longrightarrow \R$ be such that 
\begin{description}
    \item[(i)] $\mathscr L\in \mathscr C^0(\GStar{m+1} \times T\R^d, \R)$ 
    \item[(ii)] $\mathscr L_j\=\mathscr L|_{I_j} \in \mathscr C^2(I_j\times T\R^d, \R)$ for every $j=1, \ldots, m$.
\end{description}
\item We assume that condition (L1) holds\footnote{So the Lagrangian is Legendre convex and in particular the Morse index is finite.}
\item We consider the Bolza problem (meaning that fixed endpoints at the $m$ leaf vertices)  and Kirchhoff on the central vertex.
\end{itemize}
We denote by $\iMor(x, \GStar{m+1})$ (resp. $\iCLM(x, \GStar{m+1})$) the {\sc Morse index} (resp. the Maslov index) of the critical point of the Lagrangian action functional on $W^{1,2}_W(\GStar{m+1})$ where $W$ is the subspace of the configuration space corresponding to the boundary conditions. Then the following result holds.
\begin{thm}\label{thm:main-star}
    Let us consider the star graph $\GStar{m+1}$ and a Lagrangian function  $\mathscr L: \GStar{m+1} \times T\R^d \longrightarrow \R$ satisfying the assumption above and we fix an orientation on the graph. Then the following formula holds  
\[
\iMor\left(x, \GStar{m+1}\right)-\iCLM(x, \GStar{m+1})=-(m-1)\, d.
\]
\end{thm}
\begin{proof}
We start by orienting  all leaves outward and we observe that the boundary conditions on the variations corresponding to the Bolza problem at the leaf vertices are Dirichlet whilst at the center vertex we have a Kirchhoff boundary condition. Under the regularity assumptions on the Lagrangian, the action functional is of class $\mathscr C^2$. Let $x$ be a critical point and let us consider the index form at $x$. By the Legendre convexity condition, we get that the index form is a quadratic form bounded below and having finite Morse index.  By Equation~\eqref{eq:Morse-Maslov -star-graphs}, we get that 
\[
\iMor\left(x, \GStar{m+1}\right)-\iCLM(x, \GStar{m+1})=-md +\itriple\left(\Graph(\Id), \Lambda_{0}, \Lambda_{D}\right).
\]
The conclusion follows by Equation~\eqref{eq:triple-and-V} and Equation~\eqref{eq:contribution}. This concludes the proof. 
\end{proof}

%%%%%%%%%%%%%%%%%%%%%%%%%%%%%%%%%%%%%%%%%%%%%%%%%%%%%%%%%%%%%%%%%
%%
%%
%%
%%
%%
%%
%%
%%%%%%%%%%%%%%%%%%%%%%%%%%%%%%%%%%%%%%%%%%%%%%%%%%%%%%%%%%%%%%%%%

\subsection{The Morse index theorem for a two-star graph}

Now we assume that we have a graph $\Gr$ which which is the union of two star-graphs $\Gr_A$ and $\Gr_B$ as sketched at Figure~\ref{fig:2}.
\begin{figure}
\begin{center}
\begin{tikzpicture}[
  >=stealth,
  every node/.style={
    draw,
    circle,
    minimum size=6mm,
    inner sep=0pt
  }
]
  % Central vertices
  \node (A) at (0,0) {A};
  \node (B) at (6,0) {B};

  % Star at A: 3 leaves, outward arrows
  \foreach \i/\ang in {1/90,2/210,3/330}{
    \node (v\i) at (\ang:2cm) {$v_{\i}$};
    \draw[->, bend left=20] (A) to (v\i);
  }

  % Star at B: 4 leaves, inward arrows
  \foreach \j/\ang in {1/45,2/135,3/225,4/315}{
    \node (w\j) at ($(B)+(\ang:2cm)$) {$w_{\j}$};
    \draw[->, bend right=20] (w\j) to (B);
  }

  % Connecting edge A→B as an arc
  \draw[bend left=20] (A) to (B);
\end{tikzpicture}
  \caption{A two-star graph as which is the union of two star graphs having 3 and 4 leaf vertices, respectively}\label{fig:2}
  \end{center}
\end{figure}
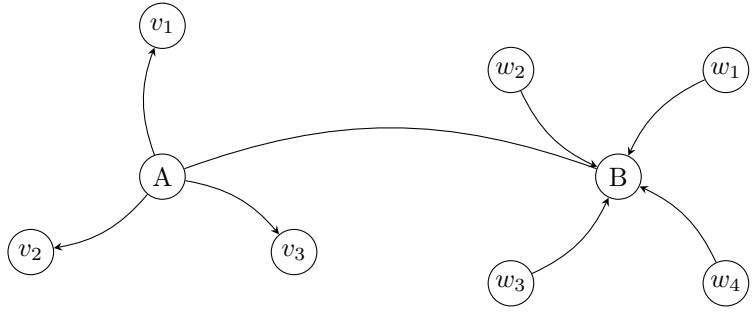
We assume that $\Gr_A=\GStar{m_A+1}$ (resp. $\Gr_B=\GStar{m_B+1}$) has $m_A$ (resp. $m_B$) leaf vertices and we assume Dirichlet boundary conditions at those $(m_A+m_B)$-vertices. So, by using the previous notation, we get  that $W_j=(0)$ for every $j=1, \ldots, (m_A+m_B)$ and so the corresponding Lagrangian subspace $\Lambda_W$ is given  by 
\[
\Lambda_W=\bigoplus_{j=1}^{m_A+m_B} L_D^j \qquad \textrm{ where } \qquad L_D^j=\R^d \times (0) \in \Lag(\R^{2d}, \omega_j). 
\]
At both of the central vertices we assume  {\sc Kirchhoff boundary conditions}. More precisely, we let 
\begin{multline}
V_A=\Set{(q_1^A, \ldots, q_{m_A+1}^A)\in (\R^d)^{m_A+1}|q_1^A=\ldots=q_{m_A+1}^A} \qquad \textrm{ and }\\[6pt]
\qquad 
V_B=\Set{(q_1^B, \ldots, q_{m_B+1}^B)\in (\R^d)^{m_B+1}|q_1^B=\ldots=q_{m_B+1}^B}
\end{multline}
and let $V_A^\perp$  and  $V_B^\perp$ respectively their  orthogonal defined by 
\begin{multline}
V_A^\perp=\Set{(q_1^A, \ldots, q_{m_A+1}^A)\in (\R^d)^{m_A+1}|q_1^A+\ldots+q_{m_A+1}^A}=0 \qquad \textrm{ and }\\[6pt]
\qquad 
V_B^\perp=\Set{(q_1^B, \ldots, q_{m_B+1}^B)\in (\R^d)^{m_B+1}|q_1^B+\ldots+q_{m_B+1}^B=0}.
\end{multline}
We define the subspace  $V\=V_A \oplus V_B$  and we associate to $V$ and its orthogonal $V^\perp$, the subspaces $I_V$ and $I_V^\perp$.  We consider the Lagrangian subspace $\Lambda_V$ defined as above, by
\[
\Lambda_V=\widetilde J(I_{V^\perp}) \oplus I_V \in \Lag(2\,d(m_A+m_B+1)).
\]
We now explicit compute the triple index appearing at Equation~\eqref{eq:Morse-Maslov -star-graphs}. To do so, we let 
\[
\widetilde V= V \oplus \underbrace{0\oplus \ldots \oplus 0}_{(m_A+m_B) \textrm {-times }} 
\]
and we observe that 
\[
\widetilde V^\perp= V^\perp \oplus \underbrace{\R^d\oplus \ldots \oplus \R^d}_{(m_A+m_B) \textrm {-times }} 
\]
By taking into account the formula given at \cite[Example 2.8]{HWY20}, we get that 
\begin{multline}\label{eq:triple-and-V-TWO-STARTS}
\itriple\left(\Graph(\Id), \Lambda_{0}, \Lambda_{D}\right)= (m_A+m_B+1)\,d-\dim\left[\widetilde  V^{\perp} \bigcap \bigoplus_{j=1}^{m_A+m_B+1}\Graph(-\Id_d)\right]\\[6pt]
=\widetilde{V}^{\perp} \cap \bigoplus_{j=1}^{m_A+m_B+1} \Graph(-\Id_d) \\[6pt] 
\cong \Set{\left(q^A_1, \cdots, q^A_{m_A+1}, q_{1}^{B}, \cdots, q_{m_B+1}^{B}\right)| \sum_{j=1}^{m_A} q_{j}^{A}=0, \quad  \sum_{i=1}^{m_B} q_{i}^{B}=0,\quad  q_{m_A+1}^{A}=-q_{m_B+1}^{B}}
\end{multline}
So
\[
\dim\left[\widetilde  V^{\perp} \bigcap \bigoplus_{j=1}^{m_A+m_B+1}\Graph(-\Id_d)\right]=(m_A+ m_B-1)\, d
\]
and hence 
\begin{equation}\label{eq:indice-triplo-2-star}
\itriple\left(\Graph(\Id), \Lambda_{0}, \Lambda_{D}\right)=
2\, d
\end{equation}
In conclusion, we get that 
\begin{equation}\label{eq:Morse-Maslov-2-star}
\iMor\left(\Gr_A\cup\Gr_B\right)-\iCLM\left(\Gr_A\cup\Gr_B\right)=
-(m_A+ m_B-1)\,d.
\end{equation}
\begin{thm}\label{thm:main--two-star}
    Let us consider an oriented two-star graph $\Gr= \Gr_A \cup \Gr_B$ and a Lagrangian function  on it $\mathscr L: \Gr \times T\R^d \longrightarrow \R$ satisfying the  above regularity  assumptions and assumption (L1). Then the following formula holds  
\[
\iMor\left(\Gr_A\cup\Gr_B\right)-\iCLM\left(\Gr_A\cup\Gr_B\right)=
-(m_A+ m_B-1)\,d.
\]
\end{thm}
\begin{proof}
We start by orienting  all leaves outward and we observe that the boundary conditions on the variations corresponding to the Bolza problem at the leaf vertices are Dirichlet whilst at the center vertices we have  Kirchhoff boundary conditions. Under the regularity assumptions on the Lagrangian, the action functional is of class $\mathscr C^2$. Let $x$ be a critical point and let us consider the index form at $x$. By the Legendre convexity condition, we get that the index form is a quadratic form bounded below and having finite Morse index.  By Equation~\eqref{eq:Morse-Maslov -star-graphs}, we get that 
\[
\iMor\left(\Gr_A\cup\Gr_B\right)-\iCLM\left(\Gr_A\cup\Gr_B\right)=-(m_A+m_B+1)\,d +\itriple\left(\Graph(\Id), \Lambda_{0}, \Lambda_{D}\right).
\]
The conclusion follows by Equation~\eqref{eq:indice-triplo-2-star}. This concludes the proof. 
\end{proof}

We close this section with the following interesting example. 
\begin{ex}
   We assume that $\Gr= \Gr_A\cup \Gr_B$ where $\Gr_A=\GStar{1}$ and $\Gr_B=\GStar{1}$, too. So in this case Equation~\eqref{eq:Morse-Maslov-2-star} reduces to the following
   \begin{equation}\label{eq:Morse-Maslov-2-star-special}
\iMor\left(\Gr_A\cup \Gr_B\right)-\iCLM(\Gr_A\cup \Gr_B)=-d.
\end{equation}
It is worth noticing that the graph $\Gr_A\cup \Gr_B$ con be identified with a segment joining the two leaf vertices and with Dirichlet boundary condition. In this case the formula provided at Equation~\eqref{eq:Morse-Maslov-2-star-special} coincides with that proved by authors at \cite[Theorem 3.4, Equation (3.16) amd Equation (3.32)]{HS09}.
\end{ex}

%\todo{For Daniele}

%%%%%%%%%%%%%%%%%%%%%%%%%%%%%%%%%%%%%%%%%%%%%%%%%%%%%%%%%%%%%%%%%
%%
%%
%%
%%
%%
%%
%%
%%%%%%%%%%%%%%%%%%%%%%%%%%%%%%%%%%%%%%%%%%%%%%%%%%%%%%%%%%%%%%%%%

\appendix

\section{Spectral flow and Maslov index}\label{sec:spectral-flow-Maslov-index}

The purpose of this Section is to provide the functional analytic and symplectic  preliminaries behind  {\bf  spectral flow\/} and the {\bf   Maslov index\/}.   Our basic references are \cite{CLM94,Dui76,RS93,LZ00}
from which  we borrow some notation and definitions.

%
%========================================
%========================================

\subsection{Maslov index for pairs of Lagrangian paths}\label{sec:Maslov}
%========================================
%========================================

We denote by $(\R^{2n}, \Omega)$ be the standard symplectic space and  
let $\mathscr P(J; \R^{2n})$ the space of continuous maps 
\[
f: J \to \Set{\textrm{pairs of Lagrangian subspaces in } \R^{2n}}
\]
equipped with the compact-open topology and we recall the following definition. 
\begin{defn}\label{def:Maslov-index}
The {\em CLM-index\/} is the unique integer valued function  
\[
 \iCLM: \mathscr P(J; \R^{2n}) \to \Z
\]
which satisfies  the following properties:
\begin{multicols}{2}
\begin{itemize}
\item[(I)] {\sc  Affine Scale Invariance}
\item[(II)] {\sc Deformation Invariance relative to the Endpoints}
\item[(III)] {\sc Path Additivity}
\item[(IV)] {\sc Symplectic Additivity}
\item[(V)] {\sc Symplectic Invariance}
\item[(VI)]{\sc Normalization}.
\end{itemize}
\end{multicols}
(We refer the interested reader to \cite[Theorem 1.1]{CLM94} for the proof). 

\end{defn}

For further reference we refer the interested reader to \cite{CLM94} and references therein.  Following authors in \cite[Section 3]{LZ00}, and references therein, let us now introduce the notion of crossing form that gives an efficient  way for computing the intersection indices   in the Lagrangian Grassmannian context.  

Let $\ell$ be a $\mathscr C^1$-curve of Lagrangian subspaces 
such that 
$\ell(0)= L$ and $\dot \ell(0)=\widehat L$. Now, if  $W$ is a fixed Lagrangian subspace transversal to $L$. For  $v \in L$ and  small enough $t$, let $w(t) \in W$ be such that $v+w(t) \in \ell(t)$.  Then the  form 
\begin{equation}\label{eq:forma-Q}
 Q(L, \widehat L)[v]= \dfrac{d}{dt}\Big\vert_{t=0} \omega \big(v, w(t)\big)
\end{equation}
is independent on the choice of $W$. 
\begin{defn}\label{def:crossing-form}
Let $t \mapsto \ell(t)=(\ell_1(t), \ell_2(t))$ be a map in  	 $\mathscr P(J; \R^{2n})$. For $t \in J$, the crossing form is a quadratic form defined by 
\begin{equation}\label{eq:crossings}
\Gamma(\ell_1, \ell_2, t)= Q(\ell_1(t), \dot \ell_1(t))- 	Q(\ell_2(t), \dot \ell_2(t))\Big\vert_{\ell_1(t)\cap \ell_2(t)}
\end{equation}
 A {\em crossing instant\/} for the curve $t \mapsto \ell(t)$ is an instant $t \in J$  such that $\ell_1(t)\cap \ell_2(t)\neq \{0\}$ nontrivially. A crossing is termed {\em regular\/} if the $\Gamma(\ell_1, \ell_2, t)$ is nondegenerate. 
 \end{defn}
 We observe that  if $t$ is a crossing instant, then $
 \Gamma(\ell_1, \ell_2,t)= - \Gamma (\ell_2, \ell_1, t).$
 If  $\ell$ is {\em regular\/} meaning that  it has only regular crossings, then the $\iCLM$-index can be computed through the crossing forms, as follows 
\begin{equation}\label{eq:iclm-crossings}
 \iCLM\big(\ell_1(t), \ell_2(t), \Omega,  t \in J\big) = \coiMor\big(\Gamma(\ell_2, \ell_1, a)\big)+ 
\sum_{a<t<b} 
 \sgn\big(\Gamma(\ell_2, \ell_1, t)\big)- \iMor\big(\Gamma(\ell_2, \ell_1,b)\big)
\end{equation}
where the summation runs over all crossings $t \in (a,b)$ and $\coiMor, \iMor$  are the dimensions  of  the positive and negative spectral spaces, respectively and $\sgn\= 
\coiMor-\iMor$ is the  signature. 
(We refer the interested reader to \cite{LZ00} ). Let $L_0$ be a distinguished Lagrangian and we assume that $\ell_1(t)\equiv L_0$ for every $t \in J$. In this case we get that the crossing form at the instant $t$ provided in Equation~\eqref{eq:crossings} actually reduce to 
\begin{equation}\label{eq:forma-crossing}
 \Gamma\big(\ell_2(t), L_0, t \big)= Q|_{\ell_2(t)\cap L_0}
\end{equation}
and hence   
\begin{equation}\label{eq:iclm-crossings-2}
 \iCLM\big(L_0, \ell_2(t), \Omega,  t \in J\big) = \coiMor\big(\Gamma(\ell_2, L_0, a)\big)+ 
\sum_{a<t<b} 
 \sgn\big(\Gamma(\ell_2, L_0, t)\big)- \iMor\big(\Gamma(\ell_2, L_0,b)\big)
\end{equation}

% % % % % % % % % % % % % % % % % % % % % % % % % % % % % % % % % % % % % % % % % % % % % % % % % % % % % 
% % % % 
% % % % 
% % % % 
% % % % % % % % % % % % % % % % % % % % % % % % % % % % % % % % % % % % % % % % % % % % % % % % % % % % % 

\subsection{On the triple and H\"ormander index}

A crucial ingredient which somehow measure the difference of the relative Maslov index with respect to two different Lagrangian subspaces is given by the H\"ormader index. Such an index is also related to the difference of the triple index and to its interesting  generalization provided recently by the last author and his co-authors in \cite{ZWZ18}. 
For, we start with  the following definition of the H\"ormander index. 
\begin{defn}\label{def:hormander}(\cite[Definition 3.9]{ZWZ18})
Let $\lambda, \mu \in \mathscr C^0\big(J, \Lagr(V,\omega)\big)$ such that 
\[
\lambda(a)=\lambda_1, \quad \lambda(b)=\lambda_2 \quad  \textrm{ and } \quad \mu(a)=\mu_1, \quad \mu(b)= \mu_2.
\]
Then the {\em H\"ormander index\/} is the integer given by 
\begin{multline}
s(\lambda_1, \lambda_2; \mu_1, \mu_2)
\= 
\iCLM\big(\mu_2, \lambda(t); t \in J\big) - 
\iCLM\big(\mu_1, \lambda(t); t \in J\big) \\
=
\iCLM\big(\mu(t), \lambda_2; t \in J\big)- \iCLM\big(\mu(t), \lambda_1; t \in J\big).
\end{multline}
Compare \cite[Equation (17), pag. 736]{ZWZ18} once observing that we observe that $\iCLM(\lambda,\mu)$ corresponds to $\textrm{Mas}\{\mu,\lambda\}$ in the notation of \cite{ZWZ18}. 
\end{defn}
\paragraph{Properties of the H\"ormander index.}
We briefly recall some well-useful properties of the H\"ormander index.
\begin{itemize}
\item 	$s(\lambda_1,\lambda_2; \mu_1, \mu_2) = -s(\lambda_1,\lambda_2; \mu_2, \mu_1)$
\item $s(\lambda_1,\lambda_2; \mu_1, \mu_2) = 
-s(\mu_1, \mu_2;\lambda_1,\lambda_2) + 
\sum_{j,k \in \{1,2\}}(-1)^{j+k+1}\dim (\lambda_j \cap \mu_k)$.
\item If $\lambda_j\cap \mu_k =\{0\}$ then $s(\lambda_1,\lambda_2; \mu_1, \mu_2) = 
-s(\mu_1, \mu_2;\lambda_1,\lambda_2)$.
\end{itemize}

The H\"ormander index is computable as the difference of the two Maslov  indices each one involving  three different Lagrangian subspaces. This index is defined in terms of the local chart representation of the atlas of the Lagrangian Grassmannian manifold.
\begin{defn}\label{def:Q-Dui76}
$\alpha,\beta,\gamma \in \Lagr(V,\omega)$ and we assume that $\alpha \cap \beta=\gamma \cap \beta=(0)$. Then we define the quadratic form $Q(\alpha,\beta;\gamma): \alpha\to \R$ as follows
\[
Q(\alpha,\beta;\gamma)[u]=\omega(u, Cu) \qquad \textrm{ where } \qquad C:\alpha \to \beta \quad\textrm{ and }\quad \gamma=\Set{u+C\,u| u\in \alpha}.
\]
\end{defn}
\begin{defn}\label{def:kashi}
Let $\alpha,\beta,\gamma \in \Lagr(V,\omega)$, $\varepsilon \= \alpha \cap \beta + \beta \cap \gamma$ and let $\pi\=\pi_\varepsilon $ be the projection in the symplectic reduction of $V$ mod $\varepsilon $.   
  We term {\em triple index\/} the integer defined by
\begin{multline}\label{eq:triple}
\itriple(\alpha, \beta, \gamma)\= \coindex Q(\pi \alpha, \pi \beta; \pi \gamma)	+\dim (\alpha \cap \gamma) -\dim (\alpha\cap \beta \cap \gamma)\\[3pt]
\leq n-\dim (\alpha \cap \beta)-\dim( \beta \cap \gamma) + \dim (\alpha \cap \beta \cap \gamma).
\end{multline}
\end{defn}
\begin{rem}\label{rem:molto-utile-stima}
Definition \ref{def:kashi} is well-posed and we refer the interested reader to \cite[Lemma 2.4]{Dui76} and  \cite[Corollary 3.12 \& Lemma 3.13]{ZWZ18} for further details).  It is worth noticing that $Q(\pi \alpha, \pi \beta; \pi \gamma)$ is a quadratic form on $\pi\alpha$. Being the reduced space $V_\varepsilon $  a $2(n-\dim \varepsilon )$ dimensional subspace, it follows that inertial indices of  $Q(\pi \alpha, \pi \beta; \pi \gamma)$ are integers between $\{0, \ldots,n-\dim \varepsilon \}$.
\end{rem}
\begin{rem}
		We also observe that	for arbitrary Lagrangian subspaces $\alpha,\beta,\gamma$, the quadratic form  $Q(\alpha,\beta,\gamma)$ is well-defined and it is a quadratic form on $\alpha\cap(\beta+\gamma)$.
	Furthermore, we have 
    \[
    \coindex Q(\alpha,\beta,\gamma)=\coindex Q(\pi\alpha,\pi\beta,\pi\gamma).
    \]
    So, we can also define the triple index as 
	\[
	\itriple(\alpha, \beta, \gamma)\= \coindex Q( \alpha,  \beta;  \gamma)	+\dim (\alpha \cap \gamma) -\dim (\alpha\cap \beta \cap \gamma).
	\]

Authors in \cite[Lemma 3.2]{ZWZ18}  give a useful property for calculating such a quadratic form.
\begin{equation}\label{eq:invariance_Q}
\coindex Q(\alpha,\beta,\gamma)=\coindex Q(\beta,\gamma,\alpha)= \coindex Q(\gamma,\alpha,\beta)
=\iMor Q(\beta,\alpha,\gamma).
\end{equation}
\end{rem}
We observe that if $(\alpha,\beta)$ is a Lagrangian decomposition of $(V,\omega)$ and $\beta \cap \gamma=\{0\}$ then $\pi$ reduces to the identity and both  terms $\dim (\alpha \cap \gamma)$ and $\dim (\alpha\cap \beta \cap \gamma)$ drop down. In this way the triple index is nothing different from the the quadratic form $Q$ defining the local chart of the atlas of $\Lagr(V,\omega)$. It is possible to prove (cfr. \cite[proof of the Lemma 3.13]{ZWZ18}) that 
\begin{equation}\label{eq:kernel-dim-q-form}
	\dim(\alpha \cap \gamma) -\dim(\alpha \cap \beta \cap \gamma)= \nullity Q(\pi \alpha, \pi \beta; \pi \gamma),
\end{equation}
where we denoted by $\nullity Q$ the nullity (namely the kernel dimension of the quadratic form $Q$).
By summing up Equation \eqref{eq:triple} and Equation \eqref{eq:kernel-dim-q-form}, we finally get 
\begin{equation}\label{eq:triple-coindex-extended}
\itriple(\alpha, \beta, \gamma)= \noo{+}Q(\pi \alpha, \pi \beta; \pi \gamma)
\end{equation}
 where $\noo{+} Q$ denotes the so-called {\em extended coindex\/} or {\em generalized coindex\/} (namely the coindex plus the nullity) of the quadratic form $Q$. (Cfr. \cite[Lemma 2.4]{Dui76} for further details).

\begin{rem}\label{rem:Agrachev-comparison}
Given three Lagrangian subspaces $L_0, L_1, L_2$ \cite[Definition 1, pag. 2883]{ABB23}, authors define an index on $L_1\cap(L_0+L_2) $, named  {\em negative Maslov index} as follows
\[
i(L_0,L_1,L_2):=\iMor [m(L_0, L_1, L_2)]
\]
where $m$ is a  quadratic form. It's straightforward to check that, $m$ is related to the above introduced quadratic form $Q$ by: 
\[
m(L_0, L_1, L_2)= Q(L_0, L_2;L_1)= -Q(L_0, L_1;L_2)
\]
where the last equality follows by the symmetry properties of $Q$. By this follows that 
\[
\coiMor[Q(L_0,L_1;L_2)] = \iMor[-m(L_0,L_1,L_2)].
\]
\end{rem}

\begin{lem}\label{thm:properties}
Let $\lambda \in \mathscr C^1\big(J, \Lagr(V,\omega)\big)$. Then, for every $\mu \in \Lagr(V, \omega)$, we have 
\begin{enumerate}
	\item[\textrm{ \bf{(I)} }] $s\big(\lambda(a), \lambda(b); \lambda(a), \mu \big)= -\itriple\big(\lambda(b), \lambda(a),\mu\big)\leq 0$, 
	\item[\textrm{ \bf{(II)} }] $s\big(\lambda(a), \lambda(b); \lambda(b), \mu \big)= \itriple\big(\lambda(a), \lambda(b),\mu\big)\geq 0$.
\end{enumerate}
\end{lem}
\begin{proof}
	For the proof, we refer the interested reader to \cite[Corollary 3.16]{ZWZ18}.
\end{proof}
The next result, which is the main result of \cite{ZWZ18}, allows us to reduce the computation of the H\"ormander index to the computation of the triple index. 
\begin{prop}{\bf \cite[Theorem 1.1]{ZWZ18}\/}\label{thm:mainli} 
Let $(V,\omega)$ be a $2n$-dimensional symplectic space and let  $\lambda_1, \lambda_2, \mu_1, \mu_2 \in \Lagr(V,\omega)$. Under the above notation, we get 
\begin{equation}
s(\lambda_1, \lambda_2,\mu_1,\mu_2)=\itriple(\lambda_1,\lambda_2,\mu_2)- 	\itriple(\lambda_1,\lambda_2,\mu_1)\\
=\itriple(\lambda_1,\mu_1,\mu_2)- \itriple(\lambda_2,\mu_1,\mu_2)
\end{equation}
\end{prop}
\begin{rem}
We emphasize that no transversality conditions are assumed on the  Lagrangian subspaces in Proposition \ref{thm:mainli} 
\end{rem}

%%%%%%%%%%%%%%%%%%%%%%%%%%%%%%%%%%%%%%%%%%%%%%%%%%%%%%%%%%%%%%%%%
%%
%%
%%
%%
%%
%%
%%
%%%%%%%%%%%%%%%%%%%%%%%%%%%%%%%%%%%%%%%%%%%%%%%%%%%%%%%%%%%%%%%%%

\subsection{A quick recap on the  spectral flow}\label{sec:sf}

The aim of this section is to briefly recall the Definition and the main
properties of the spectral
flow for a continuous path of closed selfadjoint Fredholm  operators  $\mathcal{CF}^{sa}(H )$ on the Hilbert space $H$.
Our basic reference is \cite{HP17}  and references therein.

Let $H $ be a separable complex Hilbert space and let
$A: D (A) \subset H  \to H $ be  a  selfadjoint
Fredholm 
operator. By the Spectral decomposition Theorem (cf., for instance,
\cite[Chapter III,
Theorem 6.17]{Kat80}), there is an orthogonal decomposition $
 H = E_-(A)\oplus E_0(A) \oplus E_+(A),$
that reduces the operator $A$
and has the property that
\[
 \sigma(A) \cap (-\infty,0)=\sigma\big(A_{E_-(A)}\big), \quad
 \sigma(A) \cap \{0\}=\sigma\big(A_{E_0(A)}\big),\quad
 \sigma(A) \cap (0,+\infty)=\sigma\big(A_{E_+(A)}\big).
\]
\begin{defn}\label{def:essential}
Let $A \in \mathcal{CF}^{sa}(H )$. We  term $A$ {\em essentially
positive\/}
if $\sigma_{ess}(A)\subset (0,+\infty)$, {\em essentially negative\/} if
$\sigma_{ess}(A)\subset (-\infty,0)$ and finally
{\em strongly indefinite\/} respectively if $\sigma_{ess}(A) \cap (-\infty,
0)\not=
\emptyset$ and $\sigma_{ess}(A) \cap ( 0,+\infty)\not=\emptyset$.
\end{defn}
\noindent
If $\dim E_-(A)<\infty$,
we define its {\em Morse index\/}
as the integer denoted by $\iindex{A}$ and defined as $
 \iindex{A} \= \dim E_-(A).$
Given $A \in\cfsa(H )$, for  $a,b \notin
\sigma(A)$ we set
\[
\mathcal P_{J}(A)\=\Real\left(\dfrac{1}{2\pi\, i}\int_\gamma
(\lambda-A)^{-1} d\, \lambda\right)
\]
where $\gamma$ is the circle of radius $\frac{b-a}{2}$ around the point
$\frac{a+b}{2}$. We recall that if
$J\subset \sigma(A)$ consists of  isolated eigenvalues of finite type then
$
 \Imm \mathcal P_{J}(A)= E_{J}(A)\= \bigoplus_{\lambda \in (a,b)}\ker
(\lambda -A);
$
(cf. \cite[Section XV.2]{GGK90}, for instance) and $0$ either belongs in the
resolvent set of $A$ or it is an isolated eigenvalue of finite multiplicity.
 The next result allows us to  define the spectral flow for
gap
continuous paths in  $\cfsa(H )$.
\begin{prop}\label{thm:cor2.3}
 Let $A_0 \in \cfsa(H )$ be fixed.
 \begin{enumerate}
  \item[(i)] There exists a positive real number $a \notin \sigma(A_0)$ and an
open neighborhood $\mathscr N \subset  \cfsa(H )$ of $A_0$ in the gap
topology such that $\pm a \notin
\sigma(A)$ for all $A \in  \mathscr N$ and the map
 \[
  \mathscr N \ni A \longmapsto \mathcal P_{[-a,a]}(A) \in \Lin(H )
 \]
is continuous and the projection $\mathcal P_{[-a,a]}(A)$ has constant finite
rank for all $t \in \mathscr N$.
 \item[(ii)] If $\mathscr N$ is a neighborhood as in (i) and $-a \leq c \leq d
\leq a$ are such that $c,d \notin
 \sigma(A)$ for all $A \in \mathscr N$, then $A \mapsto \mathcal P_{[c,d]}(A)$
is
continuous on $\mathscr N$.
 Moreover the rank of $\mathcal P_{[c,d]}(A) \in \mathscr N$ is finite and
constant.
 \end{enumerate}
\end{prop}
\begin{proof}
For the proof of this result we refer the interested reader to
\cite[Proposition 2.10]{BLP05}.
\end{proof}
Let $ A :J \to \cfsa(H )$ be a gap continuous path.  As
consequence
of Proposition \ref{thm:cor2.3}, for every $t \in J$ there exists $a>0$ and
an open
connected neighborhood $\mathscr N_{t,a} \subset \cfsa(H )$ of
$\mathcal
A(t)$
such that $\pm a \notin \sigma(A)$ for all $A\in \mathscr N_{t,a}$ and the map
$\mathscr N_{t,a} \in A \longmapsto \mathcal P_{[-a,a]}(A) \in \mathcal
B$
is continuous and hence $ \rk\left(\mathcal P_{[-a,a]}(A)\right)$ does not
depends on $A \in \mathscr N_{t,a}$. Let us consider the open covering
of the interval $J$ given by the
pre-images of the neighborhoods $\mathcal
N_{t,a}$ through $ A $ and, by choosing a sufficiently fine partition of
the interval $J$ having diameter less than the Lebesgue number
of the covering, we can find  $a=:t_0 < t_1 < \dots < t_n:=b$,
operators $T_i \in \cfsa(H )$ and
positive real numbers $a_i $, $i=1, \dots , n$ in such a way the restriction of
the path $ A $ on the
interval $[t_{i-1}, t_i]$ lies in the neighborhood $\mathscr N_{t_i, a_i}$ and
hence the
$\dim E_{[-a_i, a_i]}( A _t)$ is constant for $t \in [t_{i-1},t_i]$,
$i=1, \dots,n$.
\begin{defn}\label{def:spectral-flow-unb}
The \emph{spectral flow of $ A $} (on the interval $J$) is defined by
\[
 \spfl( A , J)\=\sum_{i=1}^N \dim\,E_{[0,a_i]}( A _{t_i})-
 \dim\,E_{[0,a_i]}( A _{t_{i-1}}) \in \Z.
\]
\end{defn}
The spectral flow as given in Definition \ref{def:spectral-flow-unb} is
well-defined
(in the sense that it is independent either on the partition or on the $a_i$)
and only depends on
the continuous path $ A $. Here We list one of the useful properties of the spectral flow.
\begin{itemize}
 \item[]  {\bf (Path Additivity)\/} If $ A _1,\mathcal
A_2: J \to
 \cfsa(H )$ are two continuous path such that
$ A _1(b)= A _2(a)$, then
 $
  \spfl( A _1 * A _2) = \spfl( A _1)+\spfl( A _2).
 $
\end{itemize}
As already observed, the spectral flow, in general,  depends on the whole path
and not
just on the ends. However, if the path has a special form, it actually depends
on the
end-points. More precisely, let  $ A  , B \in \cfsa(H )$
and let $\widetilde{ A }:J \to \cfsa(H )$ be the path
pointwise defined by $\widetilde{ A }(t)\= A + \widetilde{\mathcal
B}(t)$  where $
\widetilde{ B }$ is any continuous curve of $ A $-compact
operators parametrised on $J$
such that  $\widetilde{ B }(a)\=0$ and  $ \widetilde{ B }(b)\=
 B $. In this case,
the spectral flow depends of the
path $\widetilde A$, only on the endpoints (cfr. \cite{ZL99} and reference
therein).
\begin{rem}
 It is worth noticing that, since every operator $\widetilde{ A }(t)$ is
a compact perturbation of a
 a fixed one, the path $\widetilde{ A }$ is actually a continuous path
into $\Lin(W ; H )$,
 where $W \=D ( A )$.
\end{rem}
\begin{defn}\label{def:rel-Morse -index}(\cite[Definition 2.8]{ZL99}).
 Let $ A  , B \in \cfsa(H )$ and we assume that $ 
B$ is
 $ A $-compact (in the sense specified above). Then the
{\em  relative Morse index of the pair $ A $, $ A + B $\/}
is
defined by $
  \irel( A ,  A + B )=-\spfl(\widetilde{ A };J)$
where $\widetilde{ A }\= A + \widetilde{ B }(t)$ and where
$
\widetilde{ B }$ is any continuous curve parametrised on $J$
of $ A $-compact operators such that
$\widetilde{ B }(a)\=0$ and
$ \widetilde{ B }(b)\=  B $.
\end{defn}
\noindent
In the special case in which the Morse index of both operators $ A $ and
$ A + B $ are
finite, then
\begin{equation}\label{eq:miserve}
\irel( A ,  A + B )=\iMor(A  +
B)-\iMor(A).
\end{equation}

Let  $W , H $ be separable Hilbert spaces with a dense and
continuous
inclusion $W  \hookrightarrow H $ and let
$ A :J \to \cfsa(H )$  having fixed domain $W $. We
assume that $ A $ is
a continuously differentiable path  $ A : J \to \cfsa(H )$
and
we denote by $\dot{ A }_{\lambda_0}$ the derivative of
$ A _\lambda$ with respect to the parameter $\lambda \in J$ at
$\lambda_0$.
\begin{defn}\label{def:crossing-point}
 An instant $\lambda_0 \in J$ is called a {\em crossing instant\/} if
$\ker\,  A _{\lambda_0} \neq 0$. The
 crossing form at $\lambda_0$ is the quadratic form defined by
\begin{equation}
 \Gamma( A , \lambda_0): \ker  A _{\lambda_0} \to \R, \quad
\Gamma( A , \lambda_0)[u] =
\langle \dot{ A }_{\lambda_0}\, u, u\rangle_H .
\end{equation}
Moreover a  crossing $\lambda_0$ is called {\em regular\/}, if $\Gamma(\mathcal
A, \lambda_0)$ is nondegenerate.
\end{defn}
We recall that there exists $\varepsilon >0$ such that   $ A  +\delta \,
\Id_H $ has only regular crossings
  for almost every $\delta \in (-\varepsilon, \varepsilon)$. In the special case in which all crossings are regular, then the spectral flow
can be easily computed through  the
crossing forms. More precisely the following result  holds.
\begin{prop}\label{thm:spectral-flow-formula}
 If $ A :J \to \cfsa(W , H )$ has only regular
crossings then they are in a finite
 number and
 \[
  \spfl( A , J) = -\mathrm{\iMor}{\left[\Gamma( A ,a)\right]}+
\sum_{t_0 \in (a,b)}
  \sgn\left[\Gamma( A , t_0)\right]+
  \coiindex{\Gamma( A ,b)}
 \]
where the sum runs over all the crossing instants.
\end{prop}
\begin{proof}
The proof of this result follows by arguing as in \cite{RS95}. This concludes the proof.
\end{proof}

\section*{Conflict of Interest Statement}
The authors declare that there is no conflict of interest regarding the publication of this manuscript.

\section*{Data Availability Statement}
No data were generated or analyzed in the preparation of this manuscript.

\vskip.5truecm

\noindent

\begin{flushleft} 
Dr. Daniele Garrisi\\
School of Mathematical Sciences  \\
University of Nottingham Ningbo China\\
199 Taikang East Road\\
Ningbo, 315100, P. R. China \\
E-mail: \texttt{daniele.garrisi@nottingham.edu.cn}
\end{flushleft}

\vskip.5truecm

\begin{flushleft}
Prof. Alessandro Portaluri\\
Full professor of Mathematics
University of Turin (DISAFA)\\
Largo Paolo Braccini, 2 \\
10095 Grugliasco, Torino (Italy)\\
Website: \texttt{https://portalurialessandro.wordpress.com/}\\
E-mail: \texttt{alessandro.portaluri@unito.it}\\
\medskip
Visiting Professor at 
New York University Abu Dhabi\\
Saadiyat Marina District - Abu Dhabi (UAE)\\
E-mail: \texttt{ap9453@nyu.edu}\\
\end{flushleft}

\vskip1truecm

\begin{flushleft} 
Prof. Li Wu \\
School of Mathematics  \\
Shandong University\\
Jinan, 250100, P. R. China \\
E-mail: \texttt{vvvli@sdu.edu.cn}
\end{flushleft}

\end{document}